%% file: ChernCartan.tex
\documentclass{amsart}
\usepackage{ifthen}
\usepackage{amssymb}
\usepackage{euscript}
\usepackage[all]{xy}
\xyoption{web}
\usepackage{varioref}
\usepackage{graphicx}
\usepackage{verbatim}
\usepackage{lscape}
\usepackage{epsfig}
\usepackage{xspace}
\usepackage{boxedminipage}
\usepackage{ifthen}
\usepackage{subfigure}
\usepackage{longtable}
\usepackage{booktabs}
\newboolean{pdfOutput}
\setboolean{pdfOutput}{false}
\newboolean{abridged}
\setboolean{abridged}{false}
\ifthenelse{\boolean{pdfOutput}}
{
}
{
\usepackage{epsfig}
}
\newcommand{\drawfile}[1]
{
  \ifthenelse{\boolean{pdfOutput}}
  {
        \epsfig{figure={#1.pdf},width=2in}
  }
  {
  \epsfig{file={#1}.eps,width=2in}%
  }
}
\usepackage{rotating}
\newtheorem{theorem}{Theorem}

\newtheorem{lemma}{Lemma}
\newtheorem{proposition}{Proposition}
\theoremstyle{remark}

\newtheorem{definition}{Definition}
\newtheorem{example}{Example}
\newtheorem{remark}{Remark}
\newtheorem{Comment}{Comment}
\newcounter{remarkCounter}
\newcommand{\head}[1]{\textbf{#1}}

\newlength{\setBracketHeight}

\newcommand{\hook}{\ensuremath{\mathbin{ \hbox{\vrule height1.4pt
        width4pt depth-1pt \vrule height4pt width0.4pt depth-1pt}}}}

\newcommand{\pd}[2]{\ensuremath{\frac{\partial{#1}}{\partial{#2}}}}
\newcommand{\R}[1]{\ensuremath{\mathbb{R}^{#1}}}
\newcommand{\C}[1]{\ensuremath{\mathbb{C}^{#1}}}
\newcommand{\Q}[1]{\ensuremath{\mathbb{Q}^{#1}}}
\newcommand{\Z}[1]{\ensuremath{\mathbb{Z}^{#1}}}
\newcommand{\Ha}[1]{\ensuremath{\mathbb{H}^{#1}}}

\newcommand{\Gr}[2]{\ensuremath{\operatorname{Gr}_{#1}\mathbb{C}^{#2}}}
\newcommand{\Fl}[2]{\ensuremath{\mathbb{F}_{#1}\mathbb{C}^{#2}}}
\newcommand{\LagFl}[2]{\ensuremath{\operatorname{Lag}_{#1}\mathbb{C}^{#2}}}
\newcommand{\NullFl}[2]{\ensuremath{\operatorname{Null}_{#1}\mathbb{C}^{#2}}}

\newcommand{\SO}[1]{\operatorname{SO}\left({#1}\right)}
\newcommand{\so}[1]{\mathfrak{so}\left({#1}\right)}

\newcommand{\GL}[1]{\operatorname{GL}\left({#1}\right)}

\newcommand{\gl}[1]{\mathfrak{gl}\left({#1}\right)}
\newcommand{\SL}[1]{\operatorname{SL}\left({#1}\right)}
\newcommand{\PSL}[1]{\mathbb{P}\SL{#1}}

\newcommand{\Un}[1]{\operatorname{U}\left({#1}\right)}
\newcommand{\un}[1]{\ensuremath{\mathfrak{u}\left({#1}\right)}}
\newcommand{\SU}[1]{\operatorname{SU}\left({#1}\right)}
\newcommand{\su}[1]{\mathfrak{su}\left({#1}\right)}
\newcommand{\Symp}[1]{\ensuremath{\operatorname{Sp}\left({#1}\right)}}

\newcommand{\Lm}[2]{\ensuremath{\Lambda^{#1} \left ( {#2} \right )}}
\newcommand{\nForms}[2]{\ensuremath{\Omega^{#1} \left ( {#2} \right
    )}}
\newcommand{\Cohom}[2]{\ensuremath{ H^{#1} \left ( {#2} \right
    )}}

\DeclareMathOperator{\Ad}{Ad}

\newcommand{\trans}[1]{{}^t{#1}} \newcommand{\bari}{\bar{\imath}}
\newcommand{\barj}{\bar{\jmath}}

\DeclareMathOperator{\tr}{tr}

\renewcommand{\Im}{\operatorname{Im}}

\newcommand{\Proj}[1]{\mathbb{P}^{#1}}

\newcommand{\Aff}[1]{\mathbb{A}^{#1}}

\newcommand{\OO}[1]{
  \ensuremath{
    \mathcal{O}
    \ifthenelse{\equal{#1}{0}}
      {}
      {\left({#1}\right)}
  }
}
\newcommand{\OOp}[2]{
  \ensuremath{
    \mathcal{O}
    \ifthenelse{\equal{#1}{0}}
      {}
      {\left({#1}\right)}
    \ifthenelse{\equal{#2}{1}}
      {}
      {^{\oplus{#2}}}
  }
}

\newcommand{\extraSection}[2]
{
\ifthenelse{\boolean{abridged}}
  {
  }
  {
    \section{#1}
    \begin{center}
    \emph{This section will not be referred to
    subsequently, and may be skipped.}
    \end{center}
    \par{#2}
  }
}
\newcommand{\extraSubsection}[2]
{
\ifthenelse{\boolean{abridged}}
  {
  }
  {
    \subsection{#1}
    \begin{center}
    \emph{This subsection will not be referred to
    subsequently, and may be skipped.}
    \end{center}
    \par{#2}
  }
}
\newcommand{\extraStuff}[1]
{
\ifthenelse{\boolean{abridged}}
  {
  }
  {
    {#1}
  }
}
\def\cprime{$'$}

\newcommand{\Gtot}{\ensuremath{G}}
\newcommand{\gtot}{\mathfrak{g}}
\newcommand{\Gstruc}{P}
\newcommand{\gstruc}{\mathfrak{p}}
\newcommand{\SSPart}{L}
\newcommand{\sspart}{\mathfrak{l}}
\newcommand{\AbPart}{A}
\newcommand{\abpart}{\mathfrak{a}}
\newcommand{\RedPart}{LA}
\newcommand{\redpart}{\mathfrak{l} \oplus \mathfrak{a}}
\newcommand{\NilPart}{N}
\newcommand{\nilpart}{\ensuremath{\mathfrak{n}}}
\newcommand{\nil}[1]{\ensuremath{{#1}^{\NilPart}}}
\newcommand{\semib}[1]{\ensuremath{{#1}^{\Gtot/\Gstruc}}}
\newcommand{\Cartan}{\ensuremath{H}}
\newcommand{\cartan}{\ensuremath{\mathfrak{h}}}
\newcommand{\Borel}{\ensuremath{B}}
\newcommand{\borel}{\ensuremath{\mathfrak{b}}}
\newcommand{\cpt}[1]{%
        \ifthenelse{\equal{#1}{\RedPart} \or \equal{#1}{\redpart}}
        {\ensuremath{(#1)^c}}
        {\ensuremath{#1^c}}
}
\newcommand{\red}[1]{\ensuremath{{#1}^{\RedPart}}}

\DeclareMathOperator{\gr}{gr}

\newcommand{\pin}{\bullet}
\newcommand{\pout}{\times}

\begin{document}

\title[Characteristic forms]{Characteristic forms of complex Cartan geometries}
\author{Benjamin McKay}
\address{University College Cork \\ National University of Ireland}
% \email{bwmckay@stpt.usf.edu}
\date{\today} % \\ MSC Primary: 57R17, 51A10; Secondary: 14N05}
%\thanks{Thanks to Ionu\c t Ciocan-Fontanine for assistance
%with Gromov--Witten invariants}
\begin{abstract}
We calculate relations on characteristic classes which are
obstructions preventing closed K\"ahler manifolds from carrying
holomorphic Cartan geometries. We apply these relations to give
global constraints on the phase spaces of complex analytic
determined and underdetermined systems of differential equations.
\end{abstract}
\maketitle
\tableofcontents
\section{Introduction}
This article cuts away the overgrown sheaf theory from the garden of
parabolic geometry. Gunning \cite{Gunning:1978} and Kobayashi and
Ochiai \cite{KobayashiOchiai:1980} p. 78 obtained relations on
characteristic classes for normal projective connections following
the Thomas theory of normal projective connections, instead of the
Cartan theory, employing sheaves instead of principal bundles.
Kobayashi and Ochiai \cite{Kobayashi/Ochiai:1981} p. 207 extended
this approach to obtain relations on characteristic classes for
certain holomorphic $G$-structures.  Cartan's theory generalizes
easily to Cartan geometries, and manages abnormalcy without extra
effort. We use the Cartan theory to study relations
between characteristic classes.

\begin{theorem}\label{theorem:QuotientRing}
The ring of characteristic classes of a Cartan geometry on a
K\"ahler manifold is a quotient of the ring of characteristic forms
of the model via an explicit ring morphism.
\end{theorem}

We compute the relations on characteristics rings of various
rational homogeneous varieties to give examples,
and explain how to employ these results to study
various types of complex analytic differential equations.

\section{Cartan geometries}
\subsection{Definition}
\begin{definition}
  A \emph{Cartan pseudogeometry} on a manifold $M$, modelled on a
  homogeneous space $\Gtot/\Gstruc$, is a principal right $\Gstruc$-bundle $E \to M$,
  (with right $\Gstruc$ action written $r_g : E \to E$ for $g \in \Gstruc$),
  with a 1-form $\omega \in \nForms{1}{E} \otimes \gtot$,
  called the \emph{Cartan pseudoconnection} (where $\gtot, \gstruc$ are
  the Lie algebras of $\Gtot,\Gstruc$), so that $\omega$ identifies each tangent
  space of $E$ with $\gtot$.  For each $A \in \gtot$,
  let $\vec{A}$ be the vector field on $E$ satisfying $\vec{A} \hook
  \omega = A$.  A Cartan pseudogeometry is called a \emph{Cartan
    geometry} (and its Cartan pseudoconnection called a \emph{Cartan
    connection}) if (1) the vectors fields
$\vec{A}$ generate the infinitesimal right action:
    \[
    \vec{A} = \left. \frac{d}{dt}r_{e^{tA}}\right|_{t=0}
    \]
    for all $A \in \gstruc$,
and (2) the Cartan connection transforms
in the adjoint representation: $r_g^* \omega = \Ad_g^{-1} \omega$
    for all $g \in \Gstruc$.
\end{definition}
If all of the manifolds, Lie groups, differential forms, and maps
above are complex analytic, we will say that $M$ bears a
\emph{complex} Cartan geometry. All Cartan geometries will
henceforth be assumed complex.

In most of our examples, $\Gtot$ will be
a reductive algebraic group and $\Gstruc$ will be a parabolic
subgroup.
\subsection{Examples}
\begin{example}[Affine connections]
Let $M$ be a complex manifold with an affine connection (i.e. a
connection on $TM$). Lets see why this is a Cartan geometry modelled
on affine space. Let $\Gtot$ be the group of affine transformations
of $\Aff{n}$, and $\Gstruc=\GL{n,\C{}}$ the linear transformations
(i.e. fixing a point). So $\Gtot/\Gstruc=\Aff{n}$. Let $n$ be the
dimension of $M$. Let $\pi : E \to M$ be the principal right
$\GL{n,\C{}}$-bundle associated to the tangent bundle. So a point
$e$ of $E$ is a basis $e=\left(e_1, e_2, \dots, e_n\right)$ for $T_m
M$ for some point $m$ of $M$. Let $\omega^i_0$ be the 1-forms on $E$
defined by
\[
v \hook \omega^i_0 e_i = \pi'(e) \cdot v.
\]
Let $\left(\omega^i_j\right)$ be the connection 1-form;
$\left(\omega^i_j\right) \in \nForms{1}{E} \otimes \gl{n,\C{}}$.

Let
\[
\omega =
\begin{pmatrix}
0 & 0 \\
\omega^i_0 & \omega^i_j
\end{pmatrix}
\in \nForms{1}{E} \otimes \gtot.
\]
In this way, affine connections (even with torsion) are Cartan geometries
modelled on affine space.
\end{example}

\begin{example}[Projective connections]
Write the standard basis of $\C{n+1}$ as $e_0, e_1, \dots, e_n$.
Let $\Gtot=\SL{n+1,\C{}}$, and $\Gstruc$ be the stabilizer of the
line $\C{} \, e_0$. Then $\Gtot/\Gstruc=\Proj{n}$.
A Cartan geometry modelled on $\Proj{n}$ is called a \emph{projective
connection} (see Cartan \cite{Cartan:1992}).
%
%There is a well known construction of projective
%connections. Suppose that a manifold $M$ is covered by open
%sets, and each open set is equipped with a holomorphic connection.
%Imagine that on overlaps of those open sets, the connections
%might not agree, but they have the same geodesics, up
%to reparameterization. Call an atlas of these charts and connections
%(modulo obvious equivalence) a \emph{projective structure}.
%For example, in the model, $\Proj{n}$
%is covered by affine charts. Each affine chart bears the
%flat connection of affine space, with affine translation
%as parallel transport. The geodesics are straight lines.
%When we glue affine charts together,
%on the overlap, the geodesics match up, although the
%linear parameterization in one chart passes through
%a linear fractional transformation into the other chart,
%so the parameterization is no longer linear.
%
%It is well known (see Cartan \cite{Cartan:1992})
%that each projective structure determines a unique
%projective connection. On the other hand, not all
%projective connections arise in this manner.
%Those which do are called \emph{normal}.
%A projective structure determines a unique
%normal projective connection, and vice versa.
\end{example}

\begin{example}[Conformal connections]
Let $\Gtot=\SO{n+2,\C{}}$, and $\Gstruc$ be the stabilizer of a null
line in $\C{n+2}$. The quotient $Q=\Gtot/\Gstruc$ is the
hyperquadric (the set of null lines in $\C{n+2}$). A Cartan geometry
modelled on the hyperquadric is called a \emph{conformal
connection}.

It is well known (Cartan \cite{Cartan:68}) that a conformal
structure on manifold of at least three dimensions determines and is
determined by a unique conformal connection. Not every conformal
connection occurs this way; those which do are called \emph{normal}.
(See \v{C}ap \cite{Cap:2005} for the definition of normalcy.)
\end{example}

\begin{example}[$2^{\text{nd}}$ order systems
of ordinary differential equations] Tanaka \cite{Tanaka:1979} showed
that every second order system of ordinary differential equations
\[
\frac{d^2 y}{dx^2} = f\left(x,y,\frac{dy}{dx}\right),
\]
with $x \in \text{open} \subset \C{}$ and $y \in \text{open} \subset
\C{n}$, determines a unique Cartan geometry on the ``phase space''
of points $(x,y,p) \in \text{open} \subset \C{2n+1}$, (where $p$
formally represents $\frac{dy}{dx}$), modelled on
$\Proj{}T\Proj{n+1}$. In the model, we think of $(x,y)$ as
coordinates of an affine chart on $\Proj{n+1}$. In those
coordinates, $\Proj{}T\Proj{n+1}$ has points $(x,y,\xi)$ with $\xi
\in \Proj{}T_{(x,y)} \C{n+1}=\Proj{n}$. Take $p$ as an affine chart
on $\Proj{n}$, so that $\xi$ is the span of $(1,p) \in \C{n+1}$,
i.e. in homogeneous coordinates $\xi=\left[1:p\right]$. The
$2^{\text{nd}}$ order system of differential equations in the model
is $dy = p\, dx, dp=0$.
\end{example}

\begin{example}[2-plane fields on 5-folds]
It is well known (see Bryant \cite{Bryant:2000}, Cartan
\cite{Cartan:30}, Gardner \cite{Gardner:1989}, Nurowski
\cite{Nurowski:2005}, Sternberg \cite{Sternberg:1983}, Tanaka
\cite{Tanaka:1979}) that any 5-dimensional manifold equipped with a
``nondegenerate'' 2-plane field bears a canonical Cartan geometry
modelled on the hyperquadric $Q^5 = G_2/\Gstruc$ where $\Gstruc$ is
the stabilizer of a null line in the (unique up to scalar)
$G_2$-invariant quadratic form on the (unique up to isomorphism)
irreducible $G_2$-module $\C{7}$.
\end{example}

\begin{example}[$3^{\text{rd}}$ order ordinary differential equations]
Sato \& Yoshikawa \cite{Sato/Yoshikawa:1998} show that any
$3^{\text{rd}}$ order ordinary differential equation
\[
\frac{d^3 y}{dx^3} = f\left(x,y,\frac{dy}{dx},\frac{d^2y}{dx^2}\right)
\]
determines on its (4-dimensional) phase space
a Cartan geometry modelled on $\SO{5,\C{}}/B$,
with $B$ the Borel subgroup. The Cartan
geometry is invariant under local contact isomorphisms
of $3^{\text{rd}}$ order ODEs.
\end{example}

\subsection{Curvature}
Take a Cartan geometry $E \to M$ with Cartan connection $\omega$,
modelled on $\Gtot/\Gstruc$. The \emph{curvature} is
\begin{align*}
  \nabla \omega &= d \omega + \frac{1}{2}
  \left[\omega,\omega\right] \\
  &= \frac{1}{2} \kappa \semib{\omega} \wedge \semib{\omega},
\end{align*}
where $\semib{\omega} = \omega  \pmod{\gstruc} \in \nForms{1}{E}\otimes
\left( \gtot/\gstruc \right)$ is the \emph{soldering form} and $\kappa :
E \to \gtot \otimes
\Lm{2}{\gtot/\gstruc}^*$. Under right
$\Gstruc$-action $r_g : E \to E$,
\[
r_g^* \omega = \Ad_g^{-1} \omega,
r_g^* \nabla \omega = \Ad_g^{-1} \nabla \omega.
\]

\begin{example}[Affine connections]
For an affine connection,
\[
d \omega + \frac{1}{2} \left[\omega,\omega \right] =
\begin{pmatrix}
0 & 0 \\
\kappa^i_{k{\ell}} & \kappa^i_{jk{\ell}}
\end{pmatrix}
\omega^k \wedge \omega^{\ell}.
\]
The quantity $\kappa^i_{k \ell}$ is the torsion of the affine
connection, while $\kappa^i_{jk \ell}$ is the curvature, in the
coframe on $T_m M$ which pulls back to $\omega^i$ on $T_e E$.
\end{example}

\section{Why we would expect relations on characteristic classes}

Clearly the $\Gtot$-invariant vector bundles (or coherent sheaves)
on any homogeneous space $\Gtot/\Gstruc$ are precisely the vector
bundles $\Gtot \times_{\Gstruc} W$, where $W$ can be any
$\Gstruc$-module. (For many homogeneous spaces, including all
compact complex homogeneous spaces, all vector bundles are
$\Gtot$-invariant.) Suppose that $E \to M$ is a Cartan geometry
modelled on $\Gtot/\Gstruc$. To each such vector bundle, we can
associate a vector bundle $E \times_{\Gstruc} W \to M$ on every
Cartan geometry modelled on $\Gtot/\Gstruc$. This mapping
\[
\left(\text{vector bundles on }\Gtot/\Gstruc\right)^{\Gtot} \to
\text{vector bundles on } M
\]
takes sums to sums, subspaces to subspaces, quotients to quotients,
exact sequences to exact sequences, tensor products to tensor
products, etc. Moreover, it takes $\Gtot \times_{\Gstruc}
\left(\gtot/\gstruc\right)$ to $TM$, so strikes all tensor bundles
on $M$. Clearly we expect to see a map on the characteristic class
ring, or at least the part of the characteristic class ring arising
from $\Gtot$-invariant vector bundles on $\Gtot/\Gstruc$.

\section{Notation for various subgroups}
We will henceforth fix a maximal complex reductive subgroup
$\RedPart \subset \Gstruc$, and assume that we have fixed a direct
sum decomposition $\gstruc = \sspart \oplus \abpart \oplus
\nilpart$, of $\RedPart$-modules, with $\sspart \oplus \abpart
\subset \gstruc$ the Lie subalgebra of $\RedPart$, and $\sspart,
\abpart, \nilpart \subset \gstruc$ semisimple, abelian and nilpotent
complex subalgebras. Also assume that we have fixed a direct sum
decomposition $\gtot = \gstruc + \gtot/\gstruc$ of
$\RedPart$-modules. We won't need to assume that there is a closed
subgroup $\NilPart \subset \Gstruc$ with Lie algebra $\nilpart$, nor
that there are closed subgroups $\SSPart, \AbPart$ with Lie algebras
$\sspart, \abpart$.

We will write $\cpt{H}$ for a maximal compact subgroup of any group
$H$, and $\cpt{\mathfrak{h}} \subset \mathfrak{h}$ for the Lie
algebra of $\cpt{H} \subset H$. We will assume that we have picked
maximal compact subgroups of $\Gstruc$ and $\RedPart$ so that
$\cpt{\RedPart} = \RedPart \cap \cpt{\Gstruc}$.
%and $\cpt{\Gstruc}  = \Gstruc \cap \cpt{\Gtot}$.

\subsection{Reductions of structure group}
If $E \to M$ is any Cartan geometry, then there is a smooth (not
necessarily holomorphic) reduction of structure group to a principal
right $\cpt{\Gstruc}$-bundle which we will call $\cpt{E} \to M$ (see
Steenrod \cite{Steenrod:1999} 12.14). Moreover, $E = \cpt{E}
\times_{\cpt{\Gstruc}} \Gstruc$. We will also consider the principal
right $\RedPart$-bundle $\red{E} = \cpt{E} \times_{\cpt{\Gstruc}}
\RedPart$. Obviously there are many choices of reduction of
structure group $\cpt{E} \subset E,$ but any choice will lead us to
the same characteristic classes.

\begin{example}[Projective connections]
For a projective connection, we can write the elements of $\gtot$ as
matrices, say
\[
\omega =
\begin{pmatrix}
\omega^0_0 & \omega^0_j \\
\omega^i_0 & \omega^i_j
\end{pmatrix}
\]
with $i, j=1, \dots, n$, and $\omega^0_0 + \omega^i_i=0$.
Elements of the subalgebra
$\gstruc \subset \gtot$ look like
\[
\begin{pmatrix}
\omega^0_0 & \omega^0_j \\
0 & \omega^i_j
\end{pmatrix}.
\]
It is helpful to split up the Lie algebra into a sum:
\[
\gtot = \gtot/\gstruc \oplus \redpart \oplus \nilpart,
\]
(\emph{not} a sum of subalgebras, but only of vector subspaces)
where the various parts are given by splitting matrices as
\[
\omega =
\begin{pmatrix}
0 & 0 \\
\omega^i_0 & 0
\end{pmatrix}
+
\begin{pmatrix}
0 & 0 \\
0 & \omega^i_j - \frac{\delta^i_j}{n} \omega^k_k
\end{pmatrix}
+
\begin{pmatrix}
\omega^0_0 & 0 \\
0 & \frac{\delta^i_j}{n} \omega^k_k
\end{pmatrix}
+
\begin{pmatrix}
0 & \omega^0_j \\
0 & 0
\end{pmatrix}.
\]
This corresponds on the group $\Gtot$ to
splitting up any matrix $g \in \Gtot$, say
\[
g =
\begin{pmatrix}
g^0_0 & g^0_j \\
g^i_0 & g^i_j
\end{pmatrix}
\]
as
\[
g =
\begin{pmatrix}
1 & g^0_m h^m_j \\
0 & \delta^i_j
\end{pmatrix}
\begin{pmatrix}
g^0_0 - g^0_{m} h^m_p g^p_0 & 0 \\
0 & g^j_k
\end{pmatrix}
\begin{pmatrix}
1 & 0 \\
h^k_q g^q_0 & \delta^k_{\ell}
\end{pmatrix}
\]
(essentially the Harish-Chandra decomposition),
where here the components are split
up in the reverse order, and the
$\redpart$ parts are combined.

We can write the Cartan connection 1-form
$\omega$ on any projective connection $E \to M$
in the same manner.
It is helpful to split up the Maurer--Cartan
1-form into pieces as
\begin{align*}
\semib{\omega} &= \left(\omega^i_0\right)
\in \nForms{1}{\Gtot} \otimes {(\gtot/\gstruc)} \\
\red{\omega} &= \left( \omega^i_j + \delta^i_j \omega^k_k \right)
\in \nForms{1}{\Gtot} \otimes \left(\redpart\right) \\
\nil{\omega} &= \left(\omega^0_j\right)
\in \nForms{1}{\Gtot} \otimes \nilpart.
\end{align*}
The 1-forms $\semib{\omega}$ are semibasic,
i.e. they vanish on the fibers of $E \to M$.

We write $\omega^i$ for $\semib{\omega}=\omega^i_0$,
$\omega_i$ for $\nil{\omega}=\omega^0_i$, and
write $\red{\omega}$ as $\gamma^i_j$.
The curvature splits up then as in table~\vref{tbl:Struc}.
\begin{table}
\begin{align*}
  \nabla \omega^i &= d \omega^i + \gamma^i_j \wedge \omega^j \\
  &= \frac{1}{2} K^i_{kl} \omega^k \wedge \omega^l \\
  \nabla \gamma^i_j &= d \gamma^i_j + \gamma^i_k \wedge \gamma^k_j -
  \left( \omega_j \delta^i_k + \omega_k \delta^i_j
  \right) \wedge \omega^k  \\
  &= \frac{1}{2} K^i_{jkl} \omega^k \wedge \omega^l \\
  \nabla \omega_i &= d \omega_i - \gamma^j_i \wedge \omega_j \\
  &= \frac{1}{2} K_{ikl} \omega^k \wedge \omega^l \\
  0 &= K^i_{jk} + K^i_{kj} \\
  0 &= K^i_{jkl} + K^i_{jlk} \\
  0 &= K_{ikl} + K_{ilk}. \\
\end{align*}
\caption{The structure equations of a projective connection}\label{tbl:Struc}
\end{table}
Looking at the equations for $\gamma^i_j$, i.e. for $\red{\omega}$,
we can define
\[
\red{\nabla} \red{\omega} =
\left( d \gamma^i_j + \gamma^i_k \wedge \gamma^k_j \right),
\]
and find immediately that
\[
\red{\nabla} \red{\omega} =
\left(
- \left( \omega_j \delta^i_k + \omega_k \delta^i_j
  \right) \wedge \omega^k
+ \frac{1}{2} K^i_{jkl} \omega^k \wedge \omega^l.
\right)
\]
It looks at first sight as if $\red{\omega}$
is a connection 1-form. But its ``curvature''
$\red{\nabla} \red{\omega}$ is \emph{not} semibasic,
since it involves $\omega_i$ terms. So $\red{\nabla}$ is
only vaguely reminiscent of a connection. Nevertheless,
we will see that
\[
- \left( \omega_j \delta^i_k + \omega_k \delta^i_j
  \right) \wedge \omega^k
\]
behaves very much like a curvature term,
and controls characteristic classes.

The subgroup $\cpt{\Gtot} = \SU{n+1}$
acts on $\Proj{n}$ transitively. On
$\SU{n+1}$,  $\omega$ lives in $\su{n+1}$,
so that $\red{\omega} \in \un{n}$, and
$\nil{\omega} = -\left(\semib{\omega}\right)^*$,
i.e. $\omega_i = -\omega^{\bari}$.
Therefore when we restrict from $\Gtot$
to $\cpt{\Gtot}$, $\red{\omega}$ becomes
a connection, and its curvature is
precisely
\begin{align*}
\red{\nabla} \red{\omega}
&= - \left( \omega_j \delta^i_k + \omega_k \delta^i_j
  \right) \wedge \omega^k  \\
&=
\left( \omega^{\barj} \delta^i_k + \omega^{\bar{k}} \delta^i_{j}
\right) \wedge \omega^k,
\end{align*}
the curvature of the Fubini--Study metric.

If we pick any reduction $\cpt{E} \subset E$ on any projective connection,
we find
\[
\red{\nabla} \red{\omega}
= - \left( \omega_j \delta^i_k + \omega_k \delta^i_j
  \right) \wedge \omega^k
+
\frac{1}{2} \kappa^i_{jk{\ell}} \omega^k \wedge \omega^{\ell}.
\]
The last term,
\[
\frac{1}{2} \kappa^i_{jk{\ell}} \omega^k \wedge \omega^{\ell},
\]
is a semibasic $\left(2,0\right)$-form. Therefore the $(1,1)$-terms
of the curvature appear only in the terms with no
$\kappa^i_{jk{\ell}}$ factor. The curvature of the connection
$\red{\omega}$ on the reduction has $(1,1)$-part looking very much
like the Fubini--Study metric, independent of the curvature of the
Cartan connection.

Keep in mind that in a general Cartan geometry, we don't have any
control on the actual $\omega_i$ terms in the reduction. So the
$(1,1)$-terms that arise might not be given by
$\omega_i=-\omega^{\bar{i}},$ but instead by some more complicated
expression. However, the main point is that the $(1,1)$-terms are
not (directly) influenced by the curvature $\kappa$ of the Cartan
geometry at each point.
\end{example}

\begin{example}[$2^{\text{nd}}$ order systems of ordinary differential equations]
We immediately see the similarity to projective connections: write
the Cartan connection 1-form
\[
\omega =
\begin{pmatrix}
\omega^0_0 & \omega^0_1 & \omega^0_j \\
\omega^1_0 & \omega^1_1 & \omega^1_j \\
\omega^i_0 & \omega^i_1 & \omega^i_j
\end{pmatrix}
\]
for $i,j=2,\dots,n+1$. Then the stabilizer
subgroup $\Gstruc \subset \Gtot=\PSL{n+1,\C{}}$
of a point of $\Proj{}T\Proj{n+1}$
has Lie algebra the image of
\[
\begin{pmatrix}
\omega^0_0 & \omega^0_1 & \omega^0_j \\
0 & \omega^1_1 & \omega^1_j \\
0 & 0 & \omega^i_j
\end{pmatrix}.
\]
So we write
\begin{align*}
\semib{\omega} &=
\left( \omega^1_0, \omega^i_0, \omega^i_1 \right) \\
\red{\omega} &=
\left( \omega^0_0, \omega^1_1, \omega^i_j \right ) \\
\nil{\omega} &=
\left( \omega^0_1, \omega^0_j, \omega^1_j \right).
\end{align*}
The structure equations of the model
give
\begin{align*}
d \omega^0_0 &=  -\omega^0_1 \wedge \omega^1_0 - \omega^0_k \wedge \omega^k_0 \\
d \omega^1_1 &=  -\omega^1_0 \wedge \omega^0_1 - \omega^1_k \wedge \omega^k_1 \\
d \omega^i_j &=  -\omega^i_0 \wedge \omega^0_j - \omega^i_1 \wedge \omega^1_j - \omega^i_k \wedge \omega^k_j. \\
\end{align*}
On the maximal compact subgroup $\cpt{\Gtot}=\SU{n+1} \subset
\Gtot=\SL{n+1,\C{}}$, we find
\[
\nil{\omega} = - \left(\semib{\omega}\right)^*,
\]
so once again we find that
\begin{align*}
\red{\nabla} \red{\omega}
&=
\begin{pmatrix}
d \omega^0_0 \\
d \omega^1_1 \\
d \omega^i_j
+ \omega^i_k \wedge \omega^k_j \\
\end{pmatrix}
\\
&=
\begin{pmatrix}
\omega^{\bar{1}}_{\bar{0}} \wedge \omega^1_0 +
\omega^{\bar{k}}_{\bar{0}} \wedge \omega^k_0 \\
\omega^{\bar{0}}_{\bar{1}} \wedge \omega^0_1
+ \omega^{\bar{k}}_{\bar{1}} \wedge \omega^k_1 \\
\omega^{\bar{0}}_{\bar{i}} \wedge \omega^0_j
+ \omega^{\bar{1}}_{\bar{i}} \wedge \omega^1_j
\end{pmatrix}
\end{align*}
is exactly the curvature of the standard
$\SU{n+1}$-invariant K\"ahler metric on
$\Proj{}T\Proj{n+1}$.

Any Cartan geometry with the same model
will have the same structure equations,
except for Cartan connection curvature terms.
These terms are $(2,0)$-terms expressed
in terms of the semibasic $(1,0)$-forms,
i.e. components of $\semib{\omega}$, so
linear combinations of
\[
\omega^1_0 \wedge \omega^i_0, \omega^1_0 \wedge \omega^i_1,
\omega^i_0 \wedge \omega^j_0, \omega^i_0 \wedge \omega^j_1,
\omega^i_1 \wedge \omega^j_1.
\]
So on any reduction $\red{E} \subset E$,
we will find
\[
\red{\nabla} \red{\omega}
+
\begin{pmatrix}
\omega^0_1 \wedge \omega^1_0 - \omega^0_k \wedge \omega^k_0 \\
\omega^1_0 \wedge \omega^0_1 - \omega^1_k \wedge \omega^k_1 \\
\omega^i_0 \wedge \omega^0_j - \omega^i_1 \wedge \omega^1_j \\
\end{pmatrix}
\]
is $(2,0),$ a complex linear multiple of the
semibasic $(2,0)$-forms.
Therefore the Cartan connection curvature
terms never influence the $(1,1)$ part
of the curvature of the reduction.
\end{example}

\begin{example}[2-plane fields on 5-folds]
On $\C{7}$, consider the 3-form
\[
\phi = dx^{567} + dx^{125} - dx^{345}
+dx^{136} + dx^{246} + dx^{147} - dx^{237}
\]
where $dx^1,\dots,dx^7$ is a basis of $\C{7*}$ and $dx^{ij} = dx^i
\wedge dx^j$ etc. We will always use this basis of $\C{7*},$ and in
so doing we are following McLean \cite{McLean:1998}. Following
Bryant \cite{Bryant:1987} we know that the group of linear
transformations of $\C{7}$ fixing $\phi$ is $G_2$. Moreover, the
subgroup preserving the real locus $\R{7} \subset \C{7}$ is
precisely the compact form of $G_2$, which we will call $\cpt{G_2}$.
For the moment, lets work with the compact form, but of course all
of the equations are identical when complexified.

The group $\cpt{G_2}$ preserves
the usual Euclidean quadratic form
\[
\sum \left(dx^i\right)^2.
\]
Write $\Ha{}$ for the quaternions,
$L_q$ for left multiplication by $q$
and $R_q$ for right multiplication by $q$.
For exterior forms $\alpha$ and $\beta$
valued in the quaternions
(or in any associative algebra),
of degrees $a$ and $b$ respectively:
\begin{align*}
L_{\alpha} \wedge \beta &= \alpha \wedge \beta \\
R_{\alpha} \wedge \beta &= \left(-1\right)^{ab} \beta \wedge \alpha \\
L_{\alpha} \wedge L_{\beta} &= L_{\alpha \wedge \beta} \\
L_{\alpha} \wedge R_{\beta} &= \left(-1\right)^{ab}
R_{\beta} \wedge L_{\alpha} \\
R_{\alpha} \wedge R_{\beta} &= \left(-1\right)^{ab} R_{\beta \wedge \alpha}.
\end{align*}
As McLean \cite{McLean:1998} explains, the Lie algebra of
$\cpt{G_2}$ can be written as
\[
\omega = \begin{pmatrix}
L_{\lambda} - R_{\rho} & \beta \\
-\trans{\beta} & L_{\rho} - R_{\rho}
\end{pmatrix}
\]
in its $\R{7}$ representation. The meaning of $\trans{\beta}$ is
that we take $\beta$, valued in $\Im \Ha{*} \otimes \Ha{}$ and use
the invariant metric $q \mapsto q\bar{q}$ on $\Ha{}$ to identify
$\Im \Ha{*} \otimes \Ha{}$ with $\Im \Ha{} \otimes \Ha{*}$.
Moreover, McLean points out that $\beta$ satisfies
\[
\beta_1 i +  \beta_2 j +  \beta_3 k = 0.\footnote{
McLean actually says that $i \beta_1 + j \beta_2 + k \beta_3=0,$
which turns out not to be true. It does not affect his results.}
\]
To put it more concretely (and make explicit what $\trans{\beta}$
means), the Maurer--Cartan 1-form is
\[
\omega =
\begin {pmatrix}
0 &
-\lambda^1+\rho^1 &
-\lambda^2+\rho^2 &
-\lambda^3+\rho^3 &
\beta^0_1&\beta^0_2 &
\beta^0_3
\\
\lambda^1-\rho^1 &
0 &
-\lambda^3-\rho^3 &
\lambda^2+\rho^2 &
\beta^1_1&\beta^1_2 &
\beta^1_3
\\
\lambda^2-\rho^2 &
\lambda^3+\rho^3 &
0 &
-\lambda^1-\rho^1 &
\beta^2_1 &
\beta^2_2 &
\beta^2_3 \\
\lambda^3-\rho^3 &
-\lambda^2-\rho^2 &
\lambda^1+\rho^1 &
0 &
\beta^3_1&\beta^3_2 &
\beta^3_3
\\
-\beta^0_1 &
-\beta^1_1 &
-\beta^2_1 &
-\beta^3_1 &
0 &
-2\,\rho^3 &
2\,\rho^2
\\
-\beta^0_2 &
-\beta^1_2 &
-\beta^2_2 &
-\beta^3_2 &
2\,\rho^3 &
0 &
-2\,\rho^1
\\
-\beta^0_3 &
-\beta^1_3 &
-\beta^2_3 &
-\beta^3_3 &
-2\,\rho^2 &
2 \,\rho^1 &
0
\end{pmatrix}
\]
with
\[
\begin{pmatrix}
\beta^1_1+\beta^2_2+\beta^3_3
\\
\beta^0_1+\beta^3_2-\beta^2_3
\\
-\beta^3_1+\beta^0_2+\beta^1_3
\\
\beta^2_1-\beta^1_2+\beta^0_3
\end{pmatrix} = 0.
\]

The Maurer--Cartan structure equation $d \omega =
-\frac{1}{2}\left[\omega,\omega\right]$ expands out to
\begin{align*}
d \lambda + \lambda \wedge \lambda &=
\left ( \beta \wedge \trans{\beta} \right )_+ \\
d \rho + \rho \wedge \rho &=
-
\left ( \beta \wedge \trans{\beta} \right )_- \\
d \beta &= - \left ( L_{\lambda} - R_{\rho} \right )
\wedge \beta
- \beta \wedge \left ( L_{\rho} - R_{\rho} \right )
\end{align*}
where, since $\beta \wedge \trans{\beta}$ is a 2-form valued in
$\so{4}=\su{2}_+ \oplus \su{2}_-$, we can split it into
\[
\beta \wedge \trans{\beta}
= \left(\beta \wedge \trans{\beta} \right)_{+}
+
\left(\beta \wedge \trans{\beta} \right )_{-}.
\]

The line spanned by $e_6 + \sqrt{-1} e_7$ in $\C{7}$ is a null line
for the quadratic form. The subgroup of $\cpt{G_2}$ preserving the
2-plane $e_6 \wedge e_7$ is $\cpt{\Gstruc}=\Symp{1} \times \Un{1}/
\pm 1$, and its Lie algebra is cut out by setting $\beta=0$ and
$\rho^2=\rho^3=0$:
\[
\begin{pmatrix}
L_{\lambda} - R_{i \rho^1} & 0 \\
0 & L_{i \rho^1} - R_{i \rho^1}
\end{pmatrix}.
\]
Therefore on $\cpt{G_2}$,
\begin{align}
\red{\nabla} \red{\omega}
&=
\red{\nabla}
\begin{pmatrix}
\lambda \\
\rho^1
\end{pmatrix}
\\
&=
\begin{pmatrix}
\left ( \beta \wedge \trans{\beta} \right )_+ \\
-2 \rho^2 \wedge \rho^3
-\frac{1}{2}
\left(
\beta^a_2 \wedge \beta^a_3
\right)\label{eqn:BlahBlah}
\end{pmatrix}.
\end{align}
Again to be more specific,
\begin{align*}
\left(\beta \wedge \trans{\beta} \right)_+
&=
\frac{i}{2}
\left(
- \beta^0_a \wedge \beta^1_a
+ \beta^3_a \wedge \beta^2_a
\right)
\\
&+
\frac{j}{2}
\left(
- \beta^0_a \wedge \beta^2_a
+ \beta^1_a \wedge \beta^3_a
\right)\\
&+
\frac{k}{2}
\left(
- \beta^0_a \wedge \beta^3_a
+ \beta^2_a \wedge \beta^1_a
\right).
\end{align*}

All of these equations hold in the complex form of $G_2$ as well.
The complex subgroup $\Gstruc$ fixing the null line through $e_6 +
\sqrt{-1} e_7$ has Lie algebra given by
\begin{align*}
\beta_2 + \beta_3 \sqrt{-1} &= 0, \\
\rho^2 + \rho^3 \sqrt{-1} &= 0.
\end{align*}

Any nondegenerate holomorphic 2-plane field on a complex 5-manifold
gives rise to a Cartan geometry modelled on $G_2/\Gstruc$, which has
structure equations as above but with curvature ``correction
terms'':
\[
d \omega + \omega \wedge \omega =
\kappa
\semib{\omega} \wedge \semib{\omega}
\]
and
\[
\semib{\omega} =
\begin{pmatrix}
\beta_2 +  \beta_3 \sqrt{-1}\\
\rho^2 + \rho^3 \sqrt{-1}
\end{pmatrix}.
\]
Again, the $(1,1)$-terms in $\red{\nabla} \red{\omega}$
are identical to the expressions of
equation~\vref{eqn:BlahBlah} (the model), while the
Cartan connection affects only the $(2,0)$-terms:
\[
\red{\nabla} \red{\omega}
=
\begin{pmatrix}
\left( \beta \wedge \trans{\beta} \right)_+ \\
-2 \rho^2 \wedge \rho^3 - \frac{1}{2} \left( \beta^a_2 \wedge \beta^a_3 \right)
\end{pmatrix}
+ \left(2,0\right) \ \text{curvature terms}.
\]
\end{example}

\section{Characteristic forms}
Take $E \to M$ any Cartan geometry,
and pick a reduction $\red{E} \subset E.$
On $\red{E}$, the Cartan connection $\omega$ splits
into
\[
\omega =
\semib{\omega} \oplus
\red{\omega} \oplus
\nil{\omega} \in \nForms{1}{\red{E}} \otimes
\left(
\gtot/\gstruc \oplus
\redpart \oplus
\nilpart
\right).
\]
These 1-forms vary under right $\RedPart$-action as the obvious
$\RedPart$-modules.
\begin{lemma}
On any reduction of structure group $\red{E} \subset E$,
$\nil{\omega}=t \semib{\omega}$, for some
$t : \red{E} \to \left(\gtot/\gstruc\right)^* \otimes_{\R{}} \nilpart$.
\end{lemma}
\begin{remark}
This $t$ is not holomorphic for the generic choice of reduction
$\red{E}$.
\end{remark}
\begin{proof}
On $\red{E} \subset E,$ we have $\vec{A} \hook \omega = A,$
for $A \in \redpart.$ So splitting
\[
\omega = \semib{\omega} \oplus \red{\omega} \oplus \nil{\omega},
\]
we have $\vec{A} \hook \omega = \vec{A} \hook \red{\omega},$ for $A
\in \redpart.$ Therefore $\vec{A} \hook \nil{\omega}=0,$ so
$\nil{\omega}$ is semibasic. Let $\pi : E \to M$ be the bundle map.
Clearly $\semib{\omega} = \omega + \gstruc$ yields isomorphisms
\[
\xymatrix{
\left(\pi^* TM\right)_e \ar[r] &
T_eE/\ker \pi'(e) \ar[r] &
\gtot/\gstruc.
}
\]
So every semibasic 1-form is a multiple of $\semib{\omega}.$
\end{proof}
\begin{remark}
The 1-form $\red{\omega}$ is a connection for the bundle $\red{E}
\to M$ (\emph{not} a Cartan connection, unless $M$ has dimension
$0$).
\end{remark}
We can calculate the curvature of this connection as
\begin{align*}
\red{\nabla} \red{\omega} &=
d \red{\omega} + \frac{1}{2} \left[\red{\omega},\red{\omega}\right] \\
&=
\red{K} \semib{\omega} \wedge \semib{\omega}.
\end{align*}
The characteristic forms of $\red{E}$ are the differential forms
\[
P\left(\red{\nabla} \red{\omega}\right) \in
\nForms{\text{even}}{M},
\]
computed from complex polynomials $P$ on the Lie algebra which are
invariant under $\RedPart$. We will write this form as
$P\left(\red{E}\right)$.

We may also consider the same Lie algebra splitting applied to
$\omega$ on $E$. This splitting will not be invariant under the
structure group $\Gstruc$, but only under $\RedPart$. Nonetheless,
we can write down differential forms
\[
P\left(E\right)=
P\left(d \red{\omega} + \frac{1}{2}\left[
\red{\omega},\red{\omega}\right]\right)
\in \nForms{\text{even}}{E}.
\]
They may refuse to descend to $M$, being only invariant under
$\RedPart$, but in general (as we will see in examples) not under
$\nilpart$.

Obviously, the inclusion $\red{E} \subset E$ pulls back $P(E)$ to
$P\left(\red{E}\right)$. It is in this sense that we will think of
$P(E)$ as a characteristic form of $E$, even though it is not
defined downstairs on $M$.

\begin{lemma}
$P(E)$ is closed.
\end{lemma}
\begin{proof}
The proof is as for the usual theory of
characteristic forms. Because $P$ is $\RedPart$-invariant,
\[
P(A)=P\left(\Ad_g A\right),
\]
for any $g \in \RedPart$.
When we polarize $P$, this gives
\[
0 = P\left(A,A, \dots, \left[B,A\right], A, \dots, A \right)
\]
for any $A, B \in \redpart$. By the Jacobi
identity for Lie algebras,
\[
\left[\red{\omega},\left[\red{\omega},\red{\omega}\right]\right]=0.
\]
We will use the expression
$\red{\nabla} \red{\omega}$ to mean
\[
\red{\nabla} \red{\omega}
=
d \red{\omega} +
        \frac{1}{2}
        \left[
                \red{\omega},\red{\omega}
        \right]
\]
even on the bundle $E$, where this expression cannot be interpreted
as a covariant derivative in any sense. Calculate
\begin{align*}
dP(E)
&=
d
\left( P
\left(
        \red{\nabla} \red{\omega}
\right)
\right)
\\
&=
d
\left(
P
\left(
        \red{\nabla} \red{\omega}, \dots,
        \red{\nabla} \red{\omega}
\right)
\right)
\\
&=
\sum
P
\left(
        \red{\nabla} \red{\omega}, \dots,
        d \red{\nabla} \red{\omega}
        \dots,
        \red{\nabla} \red{\omega}
\right)
\\
&=
\sum
P
\left(
        \red{\nabla} \red{\omega}, \dots,
        d
        \left( d \red{\omega} + \frac{1}{2} \left[\red{\omega},\red{\omega}\right]
        \right ),
        \dots,
        \red{\nabla} \red{\omega}
\right)
\\
&=
\sum
P
\left(
        \red{\nabla} \red{\omega}, \dots,
        \left[d \red{\omega},\red{\omega}\right],
        \dots,
        \red{\nabla} \red{\omega}
\right)
\\
&=
\sum
P
\left(
        \red{\nabla} \red{\omega}, \dots,
        \left[
                \red{\nabla} \red{\omega},
                \red{\omega}
        \right],
        \dots,
        \red{\nabla} \red{\omega}
\right)
\end{align*}
which vanishes by the invariance of $P$.
\end{proof}

\begin{lemma}
The cohomology class of $P\left(\red{E}\right)$ is independent of
the choice of reduction $\red{E} \subset E$.
\end{lemma}
\begin{proof}
Because $\RedPart$ contains the maximal compact subgroup of
$\Gstruc$, $\Gstruc/\RedPart$ is contractible. Therefore the bundle
$E/\RedPart \to M$ (whose smooth sections are smooth [not
necessarily holomorphic] reductions of structure group) has
contractible fibers, and so any two smooth sections are smoothly
homotopic through smooth sections. The cohomology class is the
pullback of the cohomology class of $P(E)$.
\end{proof}

Obviously characteristic forms and classes of Cartan geometries pull
back under local isomorphisms.

\begin{lemma}
For any choice of reduction of structure group $\red{E} \subset E$
(from $\Gstruc$ to $\RedPart$), the principal right
$\RedPart$-bundle $\red{E} \to M$ is a holomorphic principal bundle,
although perhaps not holomorphically embedded inside $E$.
\end{lemma}
\begin{proof}
The $(1,0)$-forms $\nil{\omega}$ on $E$ will not generally pull back
in any simple way to $\red{E}$. However, $\semib{\omega}$ and
$\red{\omega}$ will pull back to a complex coframing on $\red{E}$.
Therefore there is a unique almost complex structure on $\red{E}$
for which these forms are $(1,0)$:
\[
\begin{pmatrix}
Jv \hook \semib{\omega} \\
Jv \hook \red{\omega} \\
\end{pmatrix}
=
i
\begin{pmatrix}
v \hook \semib{\omega} \\
v \hook \red{\omega}
\end{pmatrix}.
\]

Lets look back at $E$. Because $\nilpart \subset \gtot$ is a Lie
subalgebra, $\left[\nilpart,\nilpart\right] \subset \nilpart$, so
there are no $\nil{\omega} \wedge \nil{\omega}$ terms in $d
\semib{\omega}$ or in $d \red{\omega}$. Therefore pulling back to
$\red{E}$, even though $\nil{\omega}$ might have become
$(1,0)+(0,1)$, there are no $(0,2)$-terms in $d \semib{\omega}$ or
in $d \red{\omega}$. So the almost complex structure is a complex
structure.

On $E$, the $(1,0)$-form $\semib{\omega}$ is semibasic, and at each
point of $E$ it is pulled back from a $(1,0)$-form on $M$ which
forms a complex coframing. By pullback, the same is true on
$\red{E}$. Therefore $\red{E} \to M$ is a holomorphic map.

The group action of $\RedPart$ on $\red{E}$ transforms
$\semib{\omega}$ and $\red{\omega}$ in the obvious representation,
which is complex linear, and therefore is complex analytic.
\end{proof}

\begin{remark}
The apparent miracle that a smooth submanifold $\red{E} \subset E$
should inherit a complex structure from $E$ without being a complex
submanifold, is not really so remarkable. If there is a closed
connected complex subgroup $\NilPart \subset \Gstruc$ with Lie
algebra $\nilpart \subset \gstruc$, then the composition $\red{E}
\to E \to E/\NilPart$ is a biholomorphism.

The complex vector space splitting
\[
\gstruc = \redpart \oplus \nilpart
\]
ensures that at least locally we can find a holomorphic ``local
subgroup'' transverse to $\RedPart$ in $\Gstruc$. Therefore we can
locally holomorphically construct a transversal to the
$\RedPart$-action on $E$. This transversal replaces $E/\NilPart$,
and then we see that the complex structure on $\red{E}$ varies
biholomorphically with the choice of reduction.
\end{remark}
\begin{remark}
The characteristic form $P(E)$ is a holomorphic $(p,0)$-form ($p=2
\deg P$), because the entire construction on $E$ is complex
analytic. However, $P\left(\red{E}\right)$ is a sum of forms of
various degrees $(p,q)$ with $p+q=2 \deg P$. We have little control
on the degrees which can occur. But the summands in degrees other
than $(p,p)$ will usually vanish in cohomology.
\end{remark}
\begin{proposition}[Singer \cite{Singer:1959}]
The characteristic classes of a holomorphic principal bundle on a
K\"ahler manifold are $(p,p)$-classes.
\end{proposition}
\begin{proof}
First we need:
\begin{lemma}[Singer \cite{Singer:1959}]
There is a unique connection on $\cpt{E}$ whose induced connection
on $\red{E}$ is compatible with the complex structure, i.e. so that
$\left(\red{\nabla}\right)^{(0,1)}=\bar \partial$.
\end{lemma}
\begin{proof}
Consider a holomorphic local trivialization of $\red{E}$ and a
smooth local trivialization of $\cpt{E}$, defined in the same open
set $U \subset M$. They are related by a map
\[
\left(z,g\right) \mapsto \left(z,k(z) g\right),
\]
for $g \in \cpt{\RedPart}$, and $k : U \to \RedPart$ is smooth.
Split diffeomorphically by Cartan decomposition
$\RedPart=\cpt{\RedPart} e^\Pi$, where $\Pi \subset \redpart$ is the
locus fixed under a Cartan involution. Changing the smooth
trivialization, we can replace $k(z)$ by $h(z) k(z)$ for any $h : U
\to \cpt{\RedPart}$. Therefore we can arrange that $k : U \to
e^\Pi$. If we have a connection $g^{-1} \, dg + \Ad_g^{-1}
\cpt{\theta}$ in one trivialization, and another $g^{-1} \, dg +
\Ad_g^{-1} \theta$ in the other trivialization, then they will match
up only if
\[
\theta = \Ad_k \cpt{\theta} - dk \, k^{-1}.
\]
If we write $\cpt{\theta}=\cpt{A} \, dz + \cpt{\bar{A}} \, d\bar{z}$
and $\theta=A \, dz + B \, d\bar{z}$, then
we find
\[
B = \Ad_k \cpt{\bar{A}} - \pd{k}{\bar {z}} \, k^{-1}.
\]
Therefore compatible connections are precisely those with
\[
\cpt{\bar{A}} = k^{-1} \, \pd{k}{\bar{z}},
\
\cpt{A} = \bar{k}^{-1} \, \pd{\bar{k}}{z} + \Ad_{k}^{-1} A.
\]
Solving, we find
\begin{align*}
A &= k \, \Ad_{\bar{k}}^{-1} \pd{\bar{k}}{z}
-
\pd{k}{z} \, k^{-1}\\
B &= 0 \\
\cpt{A} &= \bar{k}^{-1} \, \pd{\bar{k}}{z}.
\end{align*}

Global existence and uniqueness follows by local existence and
uniqueness.
\end{proof}
Characteristic classes are independent of choice of connection, so
we can pick the compatible connection.
\begin{lemma}
On $\red{E}$, the curvature of the compatible connection has type
$(1,1)$.
\end{lemma}
\begin{proof}
The $(0,2)$ part is $\bar\partial^2=0$, and the $(2,0)$ part is
$\partial^2=0$.
\end{proof}
Therefore the characteristic forms for this connection are
$(p,p)$-forms. On a K\"ahler manifold, the decomposition of forms
descends to a decomposition of cohomology, and a closed $(p,p)$-form
represents a $(p,p)$-cohomology class.
\end{proof}

\section{The characteristic ring}
We prove theorem~\vref{theorem:QuotientRing}.
\begin{proof}
Consider the ring morphism $P(G) \mapsto \left[P(E)\right]$.
Both rings are obtained as quotients of the ring of invariant
polynomials on $\redpart$. So we need to check that every relation
satisfied by forms on the model is satisfied in cohomology in every
such geometry.

Suppose that $P(G)=0$ in the model, for some invariant polynomial
$P$. Consider the expression $P(E)$. This agrees entirely with
$P(G)$, modulo terms involving the curvature of the Cartan geometry.
The Cartan geometry curvature is $(2,0)$, and semibasic, i.e.
multiples of $\semib{\omega}$. On $\red{E}$, $\semib{\omega}$
remains $(1,0),$ so these curvature terms remain $(2,0)$. The 1-form
$\red{\omega}$ also remains $(1,0)$, but the 1-form $\nil{\omega}$
can be $(1,0)+(0,1)$. Since $\NilPart$ is a Lie subalgebra, there
are no $\nil{\omega} \wedge \nil{\omega}$ terms in $d
\semib{\omega}$ or $d \red{\omega}.$ So $d \semib{\omega}$ and $d
\red{\omega}$ are $(2,0)+(1,1)$. The Cartan connection curvature
terms only appear in the $(2,0)$-parts. Therefore the resulting
$(p,p)$-part has no Cartan connection curvature. It must therefore
be given by replacing $\nil{\omega}$ by its $(0,1)$-part in
equations for $d \semib{\omega}$ and $d \red{\omega}$. However, up
on $\Gtot$, the expression for $P(\Gtot)$ vanishes with linearly
independent $1$-forms appearing as $\nil{\omega}$. Therefore, these
expressions must vanish on $\red{E}$.
\end{proof}
\begin{lemma}
If $W$ is a $\Gstruc$-module, we can give rise to an associated
vector bundle $E \times_{\Gstruc} W \to M$ on every Cartan geometry
modelled on some $\Gtot/\Gstruc$. If we choose $E=\Gtot$, then this
gives an isomorphism taking the representation semiring of $\Gstruc$
to the semiring of $\Gtot$-equivariant vector bundles on
$\Gtot/\Gstruc$. For any choice of Cartan geometry $E \to M$, we can
map $\Gtot$-equivariant vector bundles on $\Gtot/\Gstruc$ to vector
bundles on $M$, via $W \mapsto E \times_{\Gstruc} W$. This rule maps
characteristic forms to characteristic classes, the same ring
morphism as above.
\end{lemma}
\begin{proof}
The characteristic classes must restrict to each reduction to be
given by the restrictions of the characteristic forms.
\end{proof}
\begin{lemma}
Suppose that $E \to M$ is a Cartan geometry, and $M$ is a closed
K\"ahler manifold. Then the Chern classes of the tangent bundle of
$M$ satisfy any relations that are satisfied between the Chern forms
of $\Gtot/\Gstruc$.
\end{lemma}
\begin{proof}
The tangent bundle is $TM = E \times_{\Gstruc} W$
for $W=\gtot/\gstruc$.
\end{proof}
\begin{example}[Affine connections]
The characteristic forms of affine space all vanish. Therefore the
characteristic classes of any closed K\"ahler manifold which admits
a holomorphic affine connection vanish.
\end{example}
\begin{example}[Projective connections]
On $\Gtot=\SL{n+1,\C{}}$,
\[
\red{\nabla} \red{\omega} =
-
\left(
        \delta^i_j \omega_k
        +
        \delta^i_k \omega_j
\right) \wedge \omega^k.
\]
For example, lets look at the projective plane, $n=2$. Then
\[
\red{\nabla} \red{\omega}
=
-
\begin{pmatrix}
2 \omega_1 \wedge \omega^1 + \omega_2 \wedge \omega^2 & \omega_2 \wedge \omega^1 \\
\omega_1 \wedge \omega^2 & \omega_1 \wedge \omega^1 + 2 \omega_2 \wedge \omega^2 \\
\end{pmatrix}.
\]
The Chern forms are
\begin{align*}
c_1(T) &= \tr \left( \frac{\sqrt{-1}}{2 \pi} \red{\nabla} \red{\omega} \right) \\
&= -\frac{3 \, \sqrt{-1}}{2 \pi}
\left(
\omega_1 \wedge \omega^1 + \omega_2 \wedge \omega^2
\right) \\
c_2(T) &= \det \left( \frac{\sqrt{-1}}{2 \pi} \red{\nabla} \red{\omega} \right) \\
&= -\frac{3}{2 \pi^2} \omega_1 \wedge \omega^1 \wedge \omega_2 \wedge \omega^2 \\
&=
\frac{1}{3} c_1(T)^2.
\end{align*}
This calculation occurs entirely on $\Gtot=\SL{3,\C{}}$, but it
gives the correct relation among Chern classes for any
closed K\"ahler surface bearing a projective connection.

\begin{theorem}[Gunning \cite{Gunning:1978}]
On $\Proj{n}$, the Chern classes of the tangent bundle $T$ satisfy
\[
c_p(T) = \binom{n+1}{p} \left(\frac{c_1(T)}{n+1}\right)^p.
\]
Therefore any closed K\"ahler manifold $M$ with a holomorphic Cartan
geometry modelled on $\Proj{n}$ (a holomorphic projective
connection) must satisfy these same equations.
\end{theorem}

For \emph{normal} Cartan geometries modelled on $\Proj{n}$, also
called normal projective connections, this was proven by Gunning
\cite{Gunning:1978}. See \v{C}ap \cite{Cap:2005} for the definition
of normalcy. We have generalized Gunning's work, not requiring
normalcy, which is vital to his proof, since he works with normal
projective connections as objects in a certain sheaf cohomology.
Gunning's interpretation of normal projective connections as objects
in sheaf cohomology has no known generalization to abnormal
projective connections or more general Cartan connections. It would
appear difficult and unnatural to attempt to find such a
generalization. We are pruning (or weeding) the sheaf cohomology.

\end{example}
\begin{remark}
One could study secondary characteristic classes of flat Cartan
geometries by a similar mechanism.
\end{remark}

\section{Atiyah classes}
We recall the definition of Atiyah class: any exact sequence of
vector bundles
\[
0 \to A \to B \to C \to 0
\]
induces an exact sequence
\[
0 \to C^* \otimes A \to C^* \otimes B \to C^* \otimes C \to 0.
\]
In cohomology, we find the exact sequence
\[
\xymatrix{
\dots \ar[r] &
\Cohom{0}{C^* \otimes C} \ar[r]^{\delta} &
\Cohom{1}{C^* \otimes A} \ar[r] &
\dots.}
\]
One obvious global section of $C^* \otimes C$ is the identity map
$1_C$. Fix $A$ and $C$, and imagine looking for all of the possible
choices of vector bundles $B$ to put in the middle. There is an
obvious choice: $B=A \oplus C$, which has $\delta 1_C=0$. It turns
out that $\delta 1_C$ determines $B$ up to isomorphism.

For a complex Lie group $G$ and holomorphic principal right
$G$-bundle, $\pi : Z \to M$, on any complex manifold $M$, there is a
vector bundle on $Z$ of vertical vectors, $\ker \pi'$. At each point
$z \in Z$ in the fiber over a point $m \in M$, we find an exact
sequence of vector spaces
\[
0 \to \ker \pi'(z) \to T_Z \to T_m M \to 0.
\]
We can identify each vertical vector $v \in \ker \pi'(z)$ with a
unique vector $A \in \mathfrak{g}$ in the Lie algebra of $G$: $A$ is
the element whose infinitesimal right action on $Z$ is given by the
vector $v$. Hence $\ker \pi'(z) \cong \mathfrak{g}$. This yields an
exact sequence of vector bundles on $M$:
\[
0 \to Z \times_G \mathfrak{g} \to TZ^{G} \to TM \to 0.
\]
where $TZ^G$ is the bundle on $M$ whose local sections are
$G$-invariant vector fields on $Z$. A holomorphic connection on $Z$
is a choice of splitting of this exact sequence.

The identity element $1_{TM} \in T^*M \otimes TM$ yields an element
$a(Z)=\delta 1_{TM}$, called the \emph{Atiyah class} of $Z$. If
there is a holomorphic connection, then clearly $a(Z)=0$, since the
exact sequences split. In fact, Atiyah \cite{Atiyah:1957} proves
that the exact sequences split, and hence there is a holomorphic
connection, just when $a(Z)=0$.

For a smooth (not necessarily holomorphic) $(1,0)$-connection
$\theta$ on $Z$, define $a(\theta) = \left(\nabla^{\theta}
\theta\right)^{(1,1)}$, the $(1,1)$-part of the curvature. This has
a cohomology class $[a(\theta)] \in \Cohom{1,1}{M, Z \times_G
\mathfrak{g}}$. Atiyah \cite{Atiyah:1957} proves that this
cohomology class is independent of the choice of $(1,0)$-connection,
and corresponds to $a(Z)$ under the Dolbeault isomorphism
\[
\Cohom{1}{M, \nForms{1}{M} \otimes \left(Z \times_G \mathfrak{g} \right)}
\to
\Cohom{1,1}{M, Z \times_G \mathfrak{g}}.
\]

Let $\Gstruc \subset \Gtot$ be a closed complex subgroup of a
complex Lie group. Let $\NilPart \subset \Gstruc$ be a closed normal
complex subgroup. For any Cartan geometry $E \to M$ modelled on
$\Gtot/\Gstruc$, we can construct the Atiyah class $a(E/\NilPart)$.

\begin{lemma}
Let $E \to M$ be a Cartan geometry with model $\Gtot/\Gstruc$,
decomposed as $\gstruc = \redpart \oplus \nilpart$. Suppose that
there is a closed normal complex Lie subgroup $\NilPart \subset
\Gstruc$, with Lie algebra $\nilpart$.  Then
\[
a\left(E/\NilPart\right)=\left(\red{\nabla} \red{\omega}\right)^{(1,1)},
\]
for any choice of reduction $\red{E} \subset E$,
under the obvious biholomorphism $\red{E} \to E \to E/\NilPart$.
\end{lemma}
\begin{proof} Atiyah \cite{Atiyah:1957} proves
that for any smooth connection of type $(1,0)$ on $E/\NilPart$, the
$(1,1)$-part of its curvature represents the Atiyah class.
\end{proof}

In this sense, we can think of $\red{\nabla} \red{\omega}$ as a kind
of Atiyah class on $E$, although we have to pullback to a reduction,
and then compute the $(1,1)$-part, before we can compute with it.

\subsection{The associated graded}
Define the filtration of any $\Gstruc$-representation $W$ by setting
$W^{\le k}=0$ if $k \le 0$ and otherwise set $W^{\le k}$ to be those
$w \in W$ for which $\nilpart w \in W^{\le k-1}$. Let $\gr W$ be the
associated graded. Hence $\gr W$ is an $\RedPart$-representation.
For $E \to M$ any Cartan geometry, clearly the characteristic
classes of $E \times_{\Gstruc} \gr W$ are the same as those of $E
\times_{\Gstruc} W$. If there is a closed subgroup $\NilPart \subset
\Gstruc$, corresponding to the Lie algebra decomposition
$\Gstruc=\RedPart \NilPart$, then clearly $E \times_{\Gstruc} \gr W
\to M$ is pulled back from $\left(E/\NilPart\right)
\times_{\RedPart} \gr W \to M$.

\section{Parabolic geometries}
\subsection{Parabolic subgroups of semisimple Lie groups}
If $\Gtot$ is a complex semisimple Lie group, and $\Gstruc \subset
\Gtot$ is a complex Lie subgroup containing a Borel subgroup,
$\Gstruc$ is called \emph{parabolic}. Wang \cite{Wang:1954} shows
that $\Gtot/\Gstruc$ is closed, and that $\Gtot$ contains a maximal
compact subgroup $\cpt{\Gtot} \subset \Gtot$ acting transitively on
$\Gtot/\Gstruc$, with $\cpt{\Gstruc} = \Gstruc \cap \cpt{\Gtot}$.
Moreover, every parabolic subgroup admits a decomposition $P=LAN$
into a product of a semisimple, an abelian, and a unipotent (see
Knapp \cite{Knapp:2002} p. 478). A complex Lie subgroup $\Gstruc
\subset \Gtot$ of a complex semisimple Lie group is parabolic just
when $\Gtot/\Gstruc$ is a rational projective variety.

Because $\Gstruc$ contains a Borel subgroup $\Borel$, it is easy to
see that the Lie algebras $\borel \subset \gstruc \subset \gtot$ of
$\Borel \subset \Gstruc \subset \Gtot$ are each a sum of root spaces
of $\gtot$. (For proof of the statements in this paragraph, see
Fulton \& Harris \cite{Fulton/Harris:1991}). We say that a root
$\alpha$ \emph{belongs} to $\gstruc$ if its root space belongs to
$\gstruc$, and otherwise we say that $\alpha$ is a root
\emph{omitted} from $\gstruc$. Moreover, the root spaces of roots
$\alpha$ of $\gtot$ for which both $\alpha$ and $-\alpha$ have root
spaces lying in $\gstruc$ form the Lie algebra $\mathfrak{l}$ of
$\SSPart$, while the root spaces of the other roots belonging to
$\gstruc$ form $\nilpart$.

Beyond the elementary facts just stated, all we need to know about
parabolic subalgebras is in the following lemma.
\begin{lemma}
The Killing form inner product
\(
\left<\delta, \beta\right>
\)
(where $\delta$ is half the sum of omitted roots, and $beta$
any root) vanishes
just precisely for $\beta$ a root of the
maximal semisimple subalgebra $\mathfrak{l} \subset \gstruc$.
\end{lemma}
\begin{proof} Knapp \cite{Knapp:2002} p. 330, corollary 5.100
gives a completely elementary proof. We give a proof along
the same lines, to keep our exposition self-contained.
Pick $\beta$ any positive root. If $\gamma$ is any
omitted root, and $\left<\gamma,\beta\right> > 0$,
then
\[
\gamma, \gamma-\beta, \gamma-2 \, \beta, \dots, \gamma-q \, \beta = r_{\beta} \gamma
\]
is a string of roots ending in the reflection $r_{\beta}$ of $\gamma$.
To start with, $\gamma$ already contains a positive
multiple of an omitted simple negative root. Equivalently, $\gamma$
has some negative multiple of a positive simple root $\alpha_1$,
for which $-\alpha_1$ is omitted. Subtracting the positive
root $\beta$ can only make the multiple of $\alpha_1$
larger negative. Therefore the entire string consists
of omitted roots.

If we have an entire $\beta$-string of omitted roots, for a positive
root $\beta$, clearly
\begin{align*}
\left<r_{\beta} \gamma,\beta\right>
&= -\left<r_{\beta} \gamma, r_{\beta} \beta\right>
\\
&=
-\left<\gamma,\beta\right>.
\end{align*}
Therefore $\left<\gamma,\beta\right>$ cancels with $\left<r_{\beta}
\gamma,\beta\right>$ in the sum $\left<\delta,\beta\right>$. Hence
the entire string cancels out of that sum.

It follows that
\begin{equation}\label{eqn:Sum}
\left<\delta,\beta\right>=
\sum_{\gamma} \left<\gamma,\beta\right>
\end{equation}
where the sum is over omitted roots $\gamma$
for which both $\left<\gamma,\beta\right> \le 0$
and for which the other end of the $\beta$-string
through $\gamma$ is not omitted.
Of course, the terms with  $\left<\gamma,\beta\right> = 0$
cancel out too, so the sum (\ref{eqn:Sum}) is over
omitted roots $\gamma$ for which both $\left<\gamma,\beta\right> < 0$
and for which the other end of the $\beta$-string
through $\gamma$ is not omitted. In particular,
the sum (\ref{eqn:Sum}) is a sum of negative terms.
But there might not be any terms.

If $\beta$ is a root of $\mathfrak{l}$, then clearly reflection in
$\beta$ preserves the roots belonging to the parabolic subgroup
$\gstruc$, and therefore preserves the omitted roots. So the omitted
roots will all lie in $\beta$-strings, and
$\left<\delta,\beta\right>=0$ for these roots. On the other hand, if
$\beta$ is not a root of $\mathfrak{l}$, then $\beta$ belongs either
to the roots of $\mathfrak{n}$, or lies among the negatives of roots
of $\mathfrak{n}$. We can assume that $\beta$ is positive, since we
only need to show that $\left<\delta,\beta\right> \ne 0$. Take
$\gamma=-\beta$, to see that the sum (\ref{eqn:Sum}) has at least
one negative term.
\end{proof}

\subsection{Characteristic forms and classes on parabolic geometries}
\begin{lemma}
Consider a parabolic subgroup $\Gstruc \subset \Gtot$ of a
semisimple complex Lie group. A relation between characteristic
classes holds on $\Gtot$ just when it holds on $\cpt{\Gtot}$, and
the same relation occurs between characteristic forms on $\Gtot$. If
$E \to M$ is a Cartan geometry modelled on $\Gtot/\Gstruc$ on a
closed K\"ahler manifold $M$, then $\left[P(\Gtot)\right] \to
\left[P(E)\right]$ is a ring morphism of the characteristic rings in
cohomology.
\end{lemma}
\begin{proof}
Suppose that $\left[P\left(\cpt{\Gtot}\right)\right]=0$ in
cohomology, so that $P\left(\cpt{\Gtot}\right)=d\xi$ for some form
$\xi$ on $\cpt{\Gtot}.$ We can average to ensure that $\xi$ is left
invariant, and invariant under right $\cpt{(\RedPart)}$-action.
Extend $\xi$ from $\cpt{\gtot}$ to a complex linear form on $\gtot.$
By Zariski density, $\xi$ extends to a left invariant holomorphic
form on $\Gtot,$ invariant under right $\RedPart$-action, and
$P\left(\cpt{\Gtot}\right)=d\xi$ on $\Gtot.$

Define differential forms $P(E)$ and $\xi(E)$ by
setting
\begin{align*}
P(E) &= P\left(\red{\nabla}\red{\omega}\right),  \\
\xi(E) &= \xi\left(\omega,\omega,\dots, \omega\right),
\end{align*}
(where the $\xi$ on the right hand side is $\xi \in
\Lm{*}{\gtot}^*$). Differentiating we see that \(P(E)-d\xi(E)\)
vanishes modulo the curvature terms. We have see that the curvature
terms in $P(E)$ appear only outside the $(p,p)$-terms. The same is
true for
\[
d\xi(E)=
\xi(d\omega,\omega,\dots,\omega)-\xi(\omega,d\omega,\dots,\omega)+\dots \pm
\xi(\omega,\omega,\dots,d\omega),
\]
as the curvature terms arise as $(2,0)$-terms inside each $d \omega$
term. Consequently if $\left[P(G)\right]=0,$ then
$\left[P(E)\right]=0.$
\end{proof}
\begin{lemma}
Let $\Borel \subset \Gtot$ be the Borel subgroup of a semisimple Lie
group. The rational cohomology of the rational homogeneous variety
$\Gtot/\Borel$ is generated by Chern classes of $\Gtot$-equivariant
line bundles $\Gtot \times_{\Borel} W$, for various 1-dimensional
$\Borel$ modules $W$.
\end{lemma}
\begin{proof}
See Bern{\v{s}}te{\u\i}n, Gel{\cprime}fand,  \&  Gel{\cprime}fand
\cite{Bernstein/Gelfand/Gelfand:1973a,Bernstein/Gelfand/Gelfand:1973b}.
\end{proof}
\begin{lemma}\label{lemma:lift}
Suppose that $\Gstruc_- \subset \Gstruc_+$ are two closed subgroups
of a Lie group $\Gtot$, so that we have a fiber bundle map
$\Gtot/\Gstruc_- \to \Gtot/\Gstruc_+$. Let $W$ be a
$\Gstruc_+$-module, and $E \to M$ a Cartan geometry modelled on
$\Gtot/\Gstruc_+$. Then $E \to E/\Gstruc_-$ is a Cartan geometry
modelled on $\Gtot/\Gstruc_-,$ called the \emph{lift} of the Cartan
geometry on $M$. The characteristic forms of $E \times_{\Gstruc_+}
W$ pull back to the characteristic forms of $E \times_{\Gstruc_-}
W$.
\end{lemma}
\begin{proof}
When we pull back vector bundles, characteristic classes pull back.
\end{proof}
\begin{lemma}
If $M$ is a closed K\"ahler manifold bearing a parabolic geometry $E
\to M$modelled on a rational homogeneous variety $\Gtot/\Gstruc,$
then the map $W \mapsto E \times_{\Gstruc} W$ on vector bundles
yields a unique ring morphism
\[
\Cohom{*}{\Gtot/\Gstruc,\Q{}} \to \Cohom{*}{M,\Q{}},
\]
sending Chern classes of $\Gtot \times_{\Gstruc} W \to E
\times_{\Gstruc} W,$ sending 1 to 1, and sending the Chern classes
of the tangent bundle of $\Gtot/\Gstruc$ to the Chern classes of the
tangent bundle of $M$.
\end{lemma}
\begin{proof}
We make lift to assume that $\Gstruc=\Borel.$ We can then restrict
attention to sums of line bundles, for which multiplicativity of the
total Chern class ensures a ring homomorphism.
\end{proof}
\begin{remark}
Note that the Chern classes live in integer cohomology. By mapping
the homogeneous vector bundles, not just Chern forms, we can try to
map the integer cohomology classes of the model to those of the
manifold, assuming the manifold $M$ is closed and K\"ahler and the
model $\Gtot/\Gstruc$ is a rational homogeneous variety. It seems
likely that the map on integer cohomology classes gives a ring
morphism $\Cohom{*}{\Gtot/\Gstruc,\Z{}} \to \Cohom{*}{M,\Z{}}.$
\end{remark}
\begin{remark} All of the nonzero cohomology
classes of $\Gtot/\Gstruc$ are $(p,p)$-classes,
so that the Hodge diamond is zero outside the middle
column.
\end{remark}

\subsection{Cohomology rings of some rational homogeneous varieties}
Baston and Eastwood \cite{Baston/Eastwood:1989} explain how to
calculate the integer cohomology rings of rational homogeneous
varieties using the Weyl group, in terms of Hasse diagrams. The
calculations are extraordinarily complicated, roughly due to the
enormous size of the Weyl group. We do not need to know the
cohomology rings of the various $\Gtot/\Gstruc$ in order to
calculate relations among characteristic classes. Therefore (just
for illustration) we will only calculate the rational cohomology
rings, and only for $\Gtot/\Gstruc$ with $\Gtot$ a rank 2 simple
group. Keep in mind that these spaces have no torsion in their
integer cohomology groups, so that there is little information lost
by working with rational coefficients. It won't be necessary for our
purposes to find the integer cohomology lattice inside the rational
cohomology.

Recall that the Weyl group of $\Gtot/\Gstruc$
means the subset of the Weyl group of $\Gtot$
leaving $\Gstruc$ invariant. This is precisely
the Weyl group of the semisimple part
of $\SSPart \subset \Gstruc$ (where $\Gstruc=\RedPart \NilPart$),
the subgroup generated by reflections in the
roots of $\SSPart$.
\begin{lemma}[Borel \cite{Borel:1953}]
The rational cohomology ring of the flag variety
$\Gtot/\Borel$ is the quotient
\[
\Cohom{*}{\Gtot/\Borel,\mathbb{Q}}
=
\operatorname{Sym}^*_{\mathbb{Q}}\left(\cartan_{\mathbb{Q}}^*\right)
/
\left(\operatorname{Sym}^+_{\mathbb{Q}}\left(\cartan_{\mathbb{Q}}^*\right)%
^{W(\Gtot)}\right).
\]
\end{lemma}
The numerator is the algebra of rational coefficient polynomials on
the rational vector space generated by the coroots. The denominator
is the ideal generated by Weyl group invariant polynomials of
positive degree. The explicit map taking a polynomial to a
cohomology class is as follows: take an element $\gamma$ of
$\cartan_{\mathbb{Z}}^*$, i.e. a weight, thought of as a 1-form on
the Cartan subgroup $\Cartan$. Any Lie algebra homomorphism $\gamma
: \borel \to \C{}$ must vanish on the nilpotent part of $\borel$, so
is determined uniquely by $\gamma : \cartan \to \C{}$. Therefore
$\gamma : \cartan \to \C{}$ determines a unique homomorphism $\gamma
: \Borel \to \C{\times},$ and so determines a holomorphic line
bundle $\Gtot \times_{\Borel} \C{}$ on $\Gtot/\Borel.$ The local
holomorphic sections of this line bundle are functions $f$ taking
open sets of $\Gtot$ to $\C{},$ so that $r_p^*f = \gamma(p)^{-1}f$,
where $r_p$ denotes right action by $p \in \Gstruc.$ We can extend
$\gamma$ to a 1-form on $\gtot$ by using the Killing form to split
apart $\gtot = \borel \oplus \borel^{\perp}$, and setting $\gamma=0$
on $\borel^{\perp}.$ We then extend $\gamma$ to a 1-form on $\Gtot$
by left invariance, making $\gamma$ a connection for this line
bundle. Taking the first Chern class:
\[
c_1(\gamma) = \frac{\sqrt{-1}}{2 \pi} d \gamma,
\]
we have associated a cohomology class to each weight. This extends
uniquely to a ring morphism.

\begin{lemma}[Borel  \cite{Borel:1953}]
The cohomology ring of $\Gtot/\Gstruc$ sits inside the cohomology ring
of $\Gtot/\Borel$, as the
set of elements invariant under the Weyl group
of $\Gtot/\Gstruc$ (i.e. under the Weyl group
of $L$, writing $\Gstruc=LAN$):
\[
\Cohom{*}{\Gtot/\Gstruc,\mathbb{Q}}=
\Cohom{*}{\Gtot/\Borel,\mathbb{Q}}^{W(\Gtot/\Gstruc)}.
\]
\end{lemma}
\begin{example}[The flag variety $\SL{3,\C{}}/B$]
If $\Gtot=\SL{3,\C{}}$, then we can label the two simple roots as
$\alpha=e_1-e_2$ and $\beta=e_2-e_3$, and then the reflections in
those roots are
\begin{alignat*}{2}
r_{\alpha}(\alpha)&=-\alpha, & \qquad r_{\alpha}(\beta)&=\alpha+\beta, \\
r_{\beta}(\beta)&=-\beta, & \qquad r_{\beta}(\alpha)&=\alpha+\beta. \\
\end{alignat*}
The cohomology ring is the quotient of the polynomial ring by the
ideal generated by the polynomials
\begin{align*}
& \alpha^2 + \alpha \beta + \beta^2, \\
& \left(\alpha-\beta\right)\left(2 \, \alpha + \beta\right)
\left(2 \, \beta + \alpha\right).
\end{align*}
Reduction via a Gr\"obner basis shows that this cohomology ring is
the one given in table~\vref{longtbl:Cohom}.

The Chern classes of the tangent bundle are computed out
by splitting the tangent bundle into a sum of line
bundles, one corresponding to each negative root.
So
\begin{align*}
c(T) &= \left(1-\alpha\right)\left(1-\beta\right)\left(1-\left(\alpha+\beta\right)\right) \\
&=
1 - 2 \left(\alpha+\beta\right) + 2 \alpha \beta + \beta^3
\end{align*}
(modulo the polynomials above).
Therefore
\begin{align*}
c_1 &= - 2 \left(\alpha+\beta\right) \\
c_2 &= 2 \alpha \beta  \\
c_3 &= \beta^3.
\end{align*}
\end{example}
\begin{example}[The projective plane]
For $\Gtot/\Gstruc=\Proj{2}$, with Dynkin diagram $\xymatrix{\pout
\ar@{-}[r] & \pin}$, (up to reordering of simple roots) $\beta$ is a
root of $\SSPart=\SL{2,\C{}}$. Moreover, $\pm \beta$ are the only
roots of $\SSPart$. The cohomology ring of $\Proj{2}$ is given by
the subring of $\Cohom{*}{G/B,\Q{}}$ invariant under reflection in
$\beta$ (see table~\vref{longtbl:Cohom}). It turns out that $\left(2
\, \alpha+\beta \right)^3$ lies in the ideal. This simple example
shows clearly how complicated the calculations get, even for the
most elementary semisimple groups and parabolic subgroups. To find
the Chern classes of the tangent bundle, we cannot split it into
line bundles, but we can lift it to a vector bundle on
$\Gtot/\Borel$, and then it splits into line bundles. We leave the
calculation to the reader.
\end{example}

%\begin{figure}
%\centering{
%\def\latticebody{%
%%\ifnum\latticeA=1 \ifnum\latticeB=-1 %
%%\else \drop{\pin}\fi\else
%%\ifnum\latticeA=0 \ifnum\latticeB=1\else
%%\drop{\pin}\fi\else\drop{\pin}\fi\fi
%\drop{}
%}
%\[
%\xy +(2,2)="o",0*\xybox{%
%%%0;<3pc,1.5mm>:<0.72pc,1.65pc>::,{"o"
%0;<1.5pc,0mm>:<0mm,1.5pc>::,{"o"
%\croplattice{-1}1{-1}1{-1}{1}{-1}1}
%,"o"+(0,0)="a"*{\pin}*+!D{}% \lambda_1}
%,"o"+(0,0)="b"*{\bigcirc}*+!L{} % \lambda_2}
%%,"o"+(0,-1)="c"*{\pin}*+!L{} % \lambda_3}
%,"o"+(1,0)="A"*{\pin}*+!D{}
%,"o"+(1,0)="AL"*{\quad \quad \alpha}*+!D{}
%,"o"+(0,1)="B"*{\pin}*+!L{}
%,"o"+(1,1)="C"*{\pin}*+!L{}
%,"o"+(-1,1)="D"*{\pin}*+!L{}
%,"o"+(-1,1.5)="DL"*{\beta}*+!L{}
%,"o"+(-1,0)="Ax"*{\pout}*+!D{}
%,"o"+(0,-1)="Bx"*{\pout}*+!L{}
%,"o"+(-1,-1)="Cx"*{\pout}*+!L{}
%,"o"+(1,-1)="D"*{\pout}*+!L{}
%%,"o"+(0,-1)="c","o"+(-1,1)="d"
%%,"a"."c"="e",!DR*{};"a"**\dir{.}
%%,"e",!UL*{};"c"**\dir{.}
%%,"b"."d"="f",!DL*{};"b"**\dir{.}
%%,"f",!UR*{};"d"**\dir{.}
%%,"e"."f"*\frm{.}
%}
%%="L","o"."L"="L",{"L"+L \ar@{-} "L"+R*+!L{}}
%%,{"L"+D \ar@{-} "L"+U*+!D{}}
%\endxy
%\]
%}
%\caption{The root spaces of the Borel subalgebra of $B_2$.}\label{fig:latticeBTwo}
%\end{figure}
\begin{example}[The varieties $\SO{5,\C{}}/\Gstruc=%
B_2/\Gstruc=C_2/\Gstruc=\Symp{4,\C{}}/\Gstruc$] As we see in
table~\vref{longtbl:Cohom}, the reflections in the two simple roots
of $B_2$ are
\begin{alignat*}{2}
r_{\alpha}(\alpha)&=-\alpha, & \qquad r_{\alpha}(\beta)&=2 \, \alpha+\beta, \\
r_{\beta}(\beta)&=-\beta, & \qquad r_{\beta}(\alpha)&=\alpha+\beta. \\
\end{alignat*}
The ideal generated by polynomials invariant under these reflections is
generated by the invariant polynomials
\begin{align*}
&\beta^2+2 \, \alpha \, \beta+2 \, \alpha^2, \\
&\beta^2 \, \left(2 \, \alpha + \beta\right)^2, \\
&\alpha^2 \, \left(\alpha+\beta\right)^2.
\end{align*}
The cohomology rings of the various $B_2/\Gstruc$ are given in
table~\vref{longtbl:Cohom}.
\end{example}

\def\latticebody{\drop{}}

\setlongtables
\begin{longtable}[c]{cccc}
\caption{Cohomologies of $G/P$ for $G$ simple of rank 2. The odd degree cohomology groups vanish.}%
\label{longtbl:Cohom} \\
\toprule
\head{Group} &
\head{Dynkin} &
\head{Root} &
\head{Cohomology} \\
& \head{diagram} & \head{lattice} \\
\midrule
\endfirsthead
\multicolumn{4}{l}{Table:~\ref{longtbl:Cohom} Cohomologies of $G/P$ for $G$ simple of rank 2 (continued)} \\ \\
\toprule
\head{Group} &
\head{Dynkin} &
\head{Root} &
\head{Cohomology} \\
& \head{diagram} & \head{lattice} \\
\midrule
\endhead
\bottomrule
\multicolumn{4}{l}{\small\sl{continues on next page}}
\endfoot
\bottomrule
\endlastfoot
$A_2$
&
{
\centering{
\xy
 0;/r.10pc/: % Set the scale to be quite small
  (0,0)*+{}="0";
  (0,0)*+{\pout}="1";
%  (10,0)*+{>}="2";
  (20,0)*+{\pout}="3";
  "1"; "3" **\dir{-};
\endxy
}
}
&
\xy +(2,2)="o",0*\xybox{%
%%0;<3pc,1.5mm>:<0.72pc,1.65pc>::,{"o"
0;<.75pc,0mm>:<0.375pc,.65pc>::,{"o"
%0;<1.5pc,0mm>:<0.75pc,1.3pc>::,{"o"
\croplattice{-4}4{-4}4{-2.6}{2.6}{-3}3}
,"o"+(0,0)="a"*{\pin}*+!D{}% \lambda_1}
,"o"+(0,0)="b"*{\bigcirc}*+!L{} % \lambda_2}
,"o"+(-2,1)="aa"*{\pout}*+!D{}% \lambda_1}
,"o"+(-1,-1)="bb"*{\pout}*+!L{} % \lambda_2}
%,"o"+(0,-1)="c"*{\pin}*+!L{} % \lambda_3}
,"o"+(1,1)="A"*{\pin}*+!D{}
%,"o"+(1,1)="AL"*{\quad \quad \quad \quad \alpha+\beta}*+!D{}
,"o"+(2,-1)="B"*{\pin}*+!L{}
,"o"+(2,-1)="BL"*{\quad \alpha}*+!L{}
,"o"+(-1,2)="C"*{\pin}*+!L{}
,"o"+(-1,2)="CL"*{\quad \beta}*+!L{}
,"o"+(1,-2)="D"*{\pout}*+!L{}
,"o"+(0,-1)="c","o"+(-1,1)="d"
%,"a"."c"="e",!DR*{};"a"**\dir{.}
%,"e",!UL*{};"c"**\dir{.}
%,"b"."d"="f",!DL*{};"b"**\dir{.}
%,"f",!UR*{};"d"**\dir{.}
%,"e"."f"*\frm{.}
}
%="L","o"."L"="L",{"L"+L \ar@{-} "L"+R*+!L{}}
%,{"L"+D \ar@{-} "L"+U*+!D{}}
\endxy
&
$
\begin{array}{cc}
%\toprule
\head{Degree} & \head{Basis} \\
\midrule
2 & \alpha, \beta \\
4 & \alpha \beta, \beta^2 \\
6 & \beta^3 \\
%\bottomrule
\end{array}
$
\\ \addlinespace[3pt]
%\cmidrule{2-4}
$A_2$
&
{
\centering{
\xy
 0;/r.10pc/: % Set the scale to be quite small
  (0,0)*+{}="0";
  (0,0)*+{\pout}="1";
%  (10,0)*+{>}="2";
  (20,0)*+{\pin}="3";
  "1"; "3" **\dir{-};
\endxy
}
}
&
\xy +(2,2)="o",0*\xybox{%
%%0;<3pc,1.5mm>:<0.72pc,1.65pc>::,{"o"
0;<.75pc,0mm>:<0.375pc,.65pc>::,{"o"
%0;<1.5pc,0mm>:<0.75pc,1.3pc>::,{"o"
\croplattice{-4}4{-4}4{-2.6}{2.6}{-3}3}
,"o"+(0,0)="a"*{\pin}*+!D{}% \lambda_1}
,"o"+(0,0)="b"*{\bigcirc}*+!L{} % \lambda_2}
,"o"+(-2,1)="aa"*{\pout}*+!D{}% \lambda_1}
,"o"+(-1,-1)="bb"*{\pout}*+!L{} % \lambda_2}
%,"o"+(0,-1)="c"*{\pin}*+!L{} % \lambda_3}
,"o"+(1,1)="A"*{\pin}*+!D{}
%,"o"+(1,1)="AL"*{\quad \quad \quad \quad \alpha+\beta}*+!D{}
,"o"+(2,-1)="B"*{\pin}*+!L{}
,"o"+(2,-1)="BL"*{\quad \alpha}*+!L{}
,"o"+(-1,2)="C"*{\pin}*+!L{}
,"o"+(-1,2)="CL"*{\quad \beta}*+!L{}
,"o"+(1,-2)="D"*{\pin}*+!L{}
,"o"+(0,-1)="c","o"+(-1,1)="d"
%,"a"."c"="e",!DR*{};"a"**\dir{.}
%,"e",!UL*{};"c"**\dir{.}
%,"b"."d"="f",!DL*{};"b"**\dir{.}
%,"f",!UR*{};"d"**\dir{.}
%,"e"."f"*\frm{.}
}
%="L","o"."L"="L",{"L"+L \ar@{-} "L"+R*+!L{}}
%,{"L"+D \ar@{-} "L"+U*+!D{}}
\endxy
&
$
\begin{array}{cc}
%\toprule
\head{Degree} & \head{Basis} \\
\midrule
2 & 2 \, \alpha + \beta \\
4 & \left(2 \, \alpha + \beta\right)^2 \\
%\bottomrule
\end{array}
$
\\
\addlinespace[3pt]
%\midrule
$B_2=C_2$
&
\begin{xy}
  0;/r.20pc/: % Set the scale to be quite small
  (0,0)*+{\pout}="1";
  (5,0)*+{>}="2";
  (10,0)*+{\pout}="3";
  "1"; "3" **\dir{=};
%  "2"; "3" **\dir{-};
% "3"; "4" **\dir{-};
% "4"; "5" **\dir{-};
  {\ar@{=}; "1";"3"};
  \end{xy}
&
\xy +(2,2)="o",0*\xybox{%
%%0;<3pc,1.5mm>:<0.72pc,1.65pc>::,{"o"
0;<1.5pc,0mm>:<0mm,1.5pc>::,{"o"
\croplattice{-1}1{-1}1{-1}{1}{-1}1}
,"o"+(0,0)="a"*{\pin}*+!D{}% \lambda_1}
,"o"+(0,0)="b"*{\bigcirc}*+!L{} % \lambda_2}
%,"o"+(0,-1)="c"*{\pin}*+!L{} % \lambda_3}
,"o"+(1,0)="A"*{\pin}*+!D{}
,"o"+(1,0)="AL"*{\quad \quad \alpha}*+!D{}
,"o"+(0,1)="B"*{\pin}*+!L{}
,"o"+(1,1)="C"*{\pin}*+!L{}
,"o"+(-1,1)="D"*{\pin}*+!L{}
,"o"+(-1,1.5)="DL"*{\beta}*+!L{}
,"o"+(-1,0)="Ax"*{\pout}*+!D{}
,"o"+(0,-1)="Bx"*{\pout}*+!L{}
,"o"+(-1,-1)="Cx"*{\pout}*+!L{}
,"o"+(1,-1)="D"*{\pout}*+!L{}
%,"o"+(0,-1)="c","o"+(-1,1)="d"
%,"a"."c"="e",!DR*{};"a"**\dir{.}
%,"e",!UL*{};"c"**\dir{.}
%,"b"."d"="f",!DL*{};"b"**\dir{.}
%,"f",!UR*{};"d"**\dir{.}
%,"e"."f"*\frm{.}
}
%="L","o"."L"="L",{"L"+L \ar@{-} "L"+R*+!L{}}
%,{"L"+D \ar@{-} "L"+U*+!D{}}
\endxy
&
$
\begin{array}{cc}
%\toprule
\head{Degree} & \head{Basis} \\
\midrule
2 & \alpha, \beta \\
4 & \alpha \beta, \beta^2 \\
6 & \alpha \beta^2, \beta^3 \\
8 & \alpha \beta^3
%\bottomrule
\end{array}
$
\\ \addlinespace[3pt]
%\cmidrule{2-4}
$B_2=C_2$
&
\begin{xy}
  0;/r.20pc/: % Set the scale to be quite small
  (0,0)*+{\pin}="1";
  (5,0)*+{>}="2";
  (10,0)*+{\pout}="3";
  "1"; "3" **\dir{=};
%  "2"; "3" **\dir{-};
% "3"; "4" **\dir{-};
% "4"; "5" **\dir{-};
  {\ar@{=}; "1";"3"};
  \end{xy}
&
\xy +(2,2)="o",0*\xybox{%
%%0;<3pc,1.5mm>:<0.72pc,1.65pc>::,{"o"
0;<1.5pc,0mm>:<0mm,1.5pc>::,{"o"
\croplattice{-1}1{-1}1{-1}{1}{-1}1}
,"o"+(0,0)="a"*{\pin}*+!D{}% \lambda_1}
,"o"+(0,0)="b"*{\bigcirc}*+!L{} % \lambda_2}
%,"o"+(0,-1)="c"*{\pin}*+!L{} % \lambda_3}
,"o"+(1,0)="A"*{\pin}*+!D{}
,"o"+(1,0)="AL"*{\quad \quad \alpha}*+!D{}
,"o"+(0,1)="B"*{\pin}*+!L{}
,"o"+(1,1)="C"*{\pin}*+!L{}
,"o"+(-1,1)="D"*{\pin}*+!L{}
,"o"+(-1,1.5)="DL"*{\beta}*+!L{}
,"o"+(-1,0)="Ax"*{\pin}*+!D{}
,"o"+(0,-1)="Bx"*{\pout}*+!L{}
,"o"+(-1,-1)="Cx"*{\pout}*+!L{}
,"o"+(1,-1)="Dx"*{\pout}*+!L{}
%,"o"+(0,-1)="c","o"+(-1,1)="d"
%,"a"."c"="e",!DR*{};"a"**\dir{.}
%,"e",!UL*{};"c"**\dir{.}
%,"b"."d"="f",!DL*{};"b"**\dir{.}
%,"f",!UR*{};"d"**\dir{.}
%,"e"."f"*\frm{.}
}
%="L","o"."L"="L",{"L"+L \ar@{-} "L"+R*+!L{}}
%,{"L"+D \ar@{-} "L"+U*+!D{}}
\endxy
&
$
\begin{array}{cc}
%\toprule
\head{Degree} & \head{Basis} \\
\midrule
2 & \alpha+\beta \\
4 & \alpha^2 \\
6 & \alpha^2 \left(\alpha+\beta\right) \\
%\bottomrule
\end{array}
$
\\ \addlinespace[3pt]
%\cmidrule{2-4}
$B_2=C_2$
&
\begin{xy}
  0;/r.20pc/: % Set the scale to be quite small
  (0,0)*+{\pout}="1";
  (5,0)*+{>}="2";
  (10,0)*+{\pin}="3";
  "1"; "3" **\dir{=};
%  "2"; "3" **\dir{-};
% "3"; "4" **\dir{-};
% "4"; "5" **\dir{-};
  {\ar@{=}; "1";"3"};
  \end{xy}
&
\xy +(2,2)="o",0*\xybox{%
%%0;<3pc,1.5mm>:<0.72pc,1.65pc>::,{"o"
0;<1.5pc,0mm>:<0mm,1.5pc>::,{"o"
\croplattice{-1}1{-1}1{-1}{1}{-1}1}
,"o"+(0,0)="a"*{\pin}*+!D{}% \lambda_1}
,"o"+(0,0)="b"*{\bigcirc}*+!L{} % \lambda_2}
%,"o"+(0,-1)="c"*{\pin}*+!L{} % \lambda_3}
,"o"+(1,0)="A"*{\pin}*+!D{}
,"o"+(1,0)="AL"*{\quad \quad \alpha}*+!D{}
,"o"+(0,1)="B"*{\pin}*+!L{}
,"o"+(1,1)="C"*{\pin}*+!L{}
,"o"+(-1,1)="D"*{\pin}*+!L{}
,"o"+(-1,1.5)="DL"*{\beta}*+!L{}
,"o"+(-1,0)="Ax"*{\pout}*+!D{}
,"o"+(0,-1)="Bx"*{\pout}*+!L{}
,"o"+(-1,-1)="Cx"*{\pout}*+!L{}
,"o"+(1,-1)="Dx"*{\pin}*+!L{}
%,"o"+(0,-1)="c","o"+(-1,1)="d"
%,"a"."c"="e",!DR*{};"a"**\dir{.}
%,"e",!UL*{};"c"**\dir{.}
%,"b"."d"="f",!DL*{};"b"**\dir{.}
%,"f",!UR*{};"d"**\dir{.}
%,"e"."f"*\frm{.}
}
%="L","o"."L"="L",{"L"+L \ar@{-} "L"+R*+!L{}}
%,{"L"+D \ar@{-} "L"+U*+!D{}}
\endxy
&
$
\begin{array}{cc}
%\toprule
\head{Degree} & \head{Basis} \\
\midrule
2 & 2 \, \alpha + \beta \\
4 & \beta^2 \\
6 & \beta^2  \left(2 \, \alpha+ \beta\right) \\
%\bottomrule
\end{array}
$
\\ \addlinespace[3pt]
%\midrule
$G_2$
&
\begin{xy}
 0;/r.10pc/: % Set the scale to be quite small
  %(0,0)*+{}="0";
  (0,0)*+{\pout}="1";
  (10,0)*+{<}="2";
  (20,0)*+{\pout}="3";
  "1"; "3" **\dir3{-};
\end{xy}
&
\xy +(2,2)="o",0*\xybox{%
%%0;<3pc,1.5mm>:<0.72pc,1.65pc>::,{"o"
0;<1.5pc,0mm>:<-2.25pc,1.3pc>::,{"o"
\croplattice{-3}3{-2}2{-2.6}{2.6}{-3}3}
,"o"+(1,0)="a"*{\pin}*+!D{}
,"o"+(1.4,0)="aL"*{\alpha}*+!D{}
,"o"+(0,0)="b"*{\bigcirc}*+!L{}
,"o"+(0,0)="c"*{\pin}*+!L{}
,"o"+(0,1)="d"*{\pin}*+!D{}
,"o"+(0,1.2)="dL"*{\beta}*+!D{}
,"o"+(1,1)="B"*{\pin}*+!L{}
,"o"+(2,1)="C"*{\pin}*+!L{}
,"o"+(3,1)="D"*{\pin}*+!L{}
,"o"+(3,2)="E"*{\pin}*+!L{}
,"o"+(-1,0)="na"*{\pout}*+!D{}% \lambda_1}
,"o"+(0,-1)="nd"*{\pout}*+!D{}
,"o"+(-1,-1)="nB"*{\pout}*+!L{}
,"o"+(-2,-1)="nC"*{\pout}*+!L{}
,"o"+(-3,-1)="nD"*{\pout}*+!L{}
,"o"+(-3,-2)="nE"*{\pout}*+!L{}
%,"o"+(3,2)="E","o"+(-1,1)="d"
%,"a"."c"="e",!DR*{};"a"**\dir{.}
%,"e",!UL*{};"c"**\dir{.}
%,"b"."d"="f",!DL*{};"b"**\dir{.}
%,"f",!UR*{};"d"**\dir{.}
%,"e"."f"*\frm{.}
}
%="L","o"."L"="L",{"L"+L \ar@{-} "L"+R*+!L{}}
%,{"L"+D \ar@{-} "L"+U*+!D{}}
\endxy
&
$
\begin{array}{cc}
%\toprule
\head{Degree} & \head{Basis} \\
\midrule
2 & \alpha, \beta \\
4 & \alpha \beta, \beta^2 \\
6 & \alpha \beta^2, \beta^3 \\
8 & \alpha \beta^3, \beta^4 \\
10 & \alpha \beta^4, \beta^5 \\
12 & \alpha \beta^5
%\bottomrule
\end{array}
$
\\ \addlinespace[3pt]
%\cmidrule{2-4}
$G_2$
&
\xy
 0;/r.10pc/: % Set the scale to be quite small
  %(0,0)*+{}="0";
  (0,0)*+{\pin}="1";
  (10,0)*+{<}="2";
  (20,0)*+{\pout}="3";
  "1"; "3" **\dir3{-};
\endxy
&
\xy +(2,2)="o",0*\xybox{%
%%0;<3pc,1.5mm>:<0.72pc,1.65pc>::,{"o"
0;<1.5pc,0mm>:<-2.25pc,1.3pc>::,{"o"
\croplattice{-3}3{-2}2{-2.6}{2.6}{-3}3}
,"o"+(1,0)="a"*{\pin}*+!D{}
,"o"+(1.4,0)="aL"*{\alpha}*+!D{}
,"o"+(0,0)="b"*{\bigcirc}*+!L{}
,"o"+(0,0)="c"*{\pin}*+!L{}
,"o"+(0,1)="d"*{\pin}*+!D{}
,"o"+(0,1.2)="dL"*{\beta}*+!D{}
,"o"+(1,1)="B"*{\pin}*+!L{}
,"o"+(2,1)="C"*{\pin}*+!L{}
,"o"+(3,1)="D"*{\pin}*+!L{}
,"o"+(3,2)="E"*{\pin}*+!L{}
,"o"+(-1,0)="na"*{\pin}*+!D{}% \lambda_1}
,"o"+(0,-1)="nd"*{\pout}*+!D{}
,"o"+(-1,-1)="nB"*{\pout}*+!L{}
,"o"+(-2,-1)="nC"*{\pout}*+!L{}
,"o"+(-3,-1)="nD"*{\pout}*+!L{}
,"o"+(-3,-2)="nE"*{\pout}*+!L{}
%,"o"+(3,2)="E","o"+(-1,1)="d"
%,"a"."c"="e",!DR*{};"a"**\dir{.}
%,"e",!UL*{};"c"**\dir{.}
%,"b"."d"="f",!DL*{};"b"**\dir{.}
%,"f",!UR*{};"d"**\dir{.}
%,"e"."f"*\frm{.}
}
%="L","o"."L"="L",{"L"+L \ar@{-} "L"+R*+!L{}}
%,{"L"+D \ar@{-} "L"+U*+!D{}}
\endxy
&
$
\begin{array}{cc}
%\toprule
\head{Degree} & \head{Basis} \\
\midrule
2 & 3 \, \alpha+2 \, \beta \\
4 & \alpha^2 \\
6 & \alpha^2 \left(3 \, \alpha+2 \, \beta\right) \\
8 & \alpha^4 \\
10 & \alpha^4 \left(3 \, \alpha + 2 \, \beta \right) \\
%\bottomrule
\end{array}
$
\\ \addlinespace[3pt]
%\cmidrule{2-4}
$G_2$
&
\xy
 0;/r.10pc/: % Set the scale to be quite small
  %(0,0)*+{}="0";
  (0,0)*+{\pout}="1";
  (10,0)*+{<}="2";
  (20,0)*+{\pin}="3";
  "1"; "3" **\dir3{-};
\endxy
&
\xy +(2,2)="o",0*\xybox{%
%%0;<3pc,1.5mm>:<0.72pc,1.65pc>::,{"o"
0;<1.5pc,0mm>:<-2.25pc,1.3pc>::,{"o"
\croplattice{-3}3{-2}2{-2.6}{2.6}{-3}3}
,"o"+(1,0)="a"*{\pin}*+!D{}
,"o"+(1.4,0)="aL"*{\alpha}*+!D{}
,"o"+(0,0)="b"*{\bigcirc}*+!L{}
,"o"+(0,0)="c"*{\pin}*+!L{}
,"o"+(0,1)="d"*{\pin}*+!D{}
,"o"+(0,1.2)="dL"*{\beta}*+!D{}
,"o"+(1,1)="B"*{\pin}*+!L{}
,"o"+(2,1)="C"*{\pin}*+!L{}
,"o"+(3,1)="D"*{\pin}*+!L{}
,"o"+(3,2)="E"*{\pin}*+!L{}
,"o"+(-1,0)="na"*{\pout}*+!D{}% \lambda_1}
,"o"+(0,-1)="nd"*{\pin}*+!D{}
,"o"+(-1,-1)="nB"*{\pout}*+!L{}
,"o"+(-2,-1)="nC"*{\pout}*+!L{}
,"o"+(-3,-1)="nD"*{\pout}*+!L{}
,"o"+(-3,-2)="nE"*{\pout}*+!L{}
%,"o"+(3,2)="E","o"+(-1,1)="d"
%,"a"."c"="e",!DR*{};"a"**\dir{.}
%,"e",!UL*{};"c"**\dir{.}
%,"b"."d"="f",!DL*{};"b"**\dir{.}
%,"f",!UR*{};"d"**\dir{.}
%,"e"."f"*\frm{.}
}
%="L","o"."L"="L",{"L"+L \ar@{-} "L"+R*+!L{}}
%,{"L"+D \ar@{-} "L"+U*+!D{}}
\endxy
&
$
\begin{array}{cc}
%\toprule
\head{Degree} & \head{Basis} \\
\midrule
2 & 2 \, \alpha + \beta \\
4 & \beta^2 \\
6 & \beta^2 \left(2 \, \alpha+\beta\right) \\
8 & \beta^4 \\
10 & \beta^4 \left(2 \, \alpha+ \beta \right)
%\bottomrule
\end{array}
$
\\
\end{longtable}

\begin{example}[$G_2/P$]
The roots of $G_2$ are drawn in table~\ref{longtbl:Cohom}.
The reflections in the two simple roots are
\begin{alignat*}{2}
r_{\alpha}\left(\alpha\right) &= -\alpha,
& \qquad r_{\alpha}\left(\beta\right)&=3\alpha + \beta, \\
r_{\beta}\left(\alpha\right) &= \alpha+\beta,
& \qquad r_{\beta}\left(\beta\right)&=-\beta.
\end{alignat*}
The ideal of Weyl invariant positive degree polynomials
is generated by the invariant polynomials:
\begin{align*}
& 3 \, \alpha^2 + 3 \, \alpha \beta + \beta^2, \\
& \beta^2 \left( 3 \, \alpha + 2 \, \beta \right)^2 \left(\beta + 3 \, \alpha \right)^2
\end{align*}
\end{example}

\subsection{Relations on characteristic classes of rational homogeneous varieties}

\begin{lemma}
Let $\Gtot$ be a complex semisimple Lie group with Borel subgroup
$\Borel \subset \Gtot$. If $W$ is a $\Borel$-module, then the
characteristic forms of the vector bundle $\Gtot \times_{\Borel} W$
are identical to those of $\Gtot \times_{\Borel} \gr W$. Moreover,
$\Gtot \times_{\Borel} \gr W$ is a sum of homogeneous line bundles,
given by the direct sum
\[
\gr W = \bigoplus_{\lambda} W_{\lambda},
\]
with $\lambda$ a weight, and $W_{\lambda}$ the associated weight
space. The total Chern class of $W$ is therefore
\[
c(W) = \prod_{\lambda} \left(1 + \frac{\sqrt{-1}}{2 \pi} d \lambda\right)^{\dim W_{\lambda}},
\]
where we treat each weight $\lambda \in \cartan^* \subset \gtot^*$
as a 1-form on $\Gtot$ before taking exterior derivative.
\end{lemma}

\subsection{Applications to parabolic geometries}

Table~\vref{longtable:ChernClassRelations} contains all rational
homogeneous varieties $\Gtot/\Gstruc$ of dimension at most 7 with
$\Gtot$ simple. I have contacted several experts on the cohomology
of rational homogeneous varieties, who inform me that the relations
among Chern classes of the tangent bundles of rational homogeneous
varieties are not known. Being unable to find a simple expression
for these relations, I offer the reader this long table to provide
the only known relations. Beside each Dynkin diagram are the ranks
of the filtration of $\gtot/\gstruc$ so that the tangent bundle of
the rational homogeneous variety is invariantly filtered by
homogeneous vector subbundles of those ranks. The final column
presents lists of polynomials in Chern classes. These polynomials
vanish. The notation is somewhat confusing, but has certain
advantages. The quantities $c_1, c_2, \dots, c_n$ are the Chern
classes of the tangent bundle of $\Gtot/\Gstruc$. For each
$\Gtot/\Gstruc$, we have found a weight $\varepsilon$ of $\gtot$,
invariant under the Weyl group of the maximal semisimple subgroup $L
\subset \Gstruc$. Therefore there is a line bundle on
$\Gtot/\Gstruc$ associated to this weight. We have expressed $c_1$
as a multiple of $\varepsilon$, i.e. as a multiple of the first
Chern class of the associated line bundle. However, if
$c_1=\varepsilon$, then we have omitted this equation. The relations
between Chern classes are far simpler when expressed in terms of
this $\varepsilon$.

The relations on Chern classes listed might not generate
all such relations. The pattern in these relations is still
unclear. The polynomial expressions
vanish in cohomology, but many of them do not vanish as
invariant differential forms, i.e.
the Chern
\emph{forms} often do \emph{not} satisfy these polynomials; only
the Chern \emph{classes} do.  All of the polynomials are
exterior derivatives of invariant differential forms.
For many $\Gtot/\Gstruc,$ all Chern classes
can be expressed in terms of $\varepsilon$ (and therefore
in terms of $c_1$). However, in the case of
\begin{xy}
  0;/r.10pc/: % Set the scale to be quite small
  (0,0)*+{\pout}="1";
  (22,0)*+{\pout}="2";
  (11,0)*+{>}="edge";
  {\ar@{=}; "1";"2"};
\end{xy}, $c_3$ is \emph{not} a polynomial function of
$c_1$ (or of $c_1$and $c_2$) in the cohomology ring. So in general,
we cannot always express Chern classes of rational
homogeneous varieties in terms of the first Chern class.

Kobayashi \cite{Kobayashi/Horst:1983} proved that
$n$-dimensional closed K\"ahler manifolds with conformal geometries satisfy
\[
c_p = a_p \left(\frac{c_1}{n}\right)^p, \ p=2,3,\dots,n,
\]
for some integers $a_p$ which remain unknown.
We have seen above similar relations for projective
connections. These are the only relations known outside
of our table.
It seems likely that a complete description of the
relations among Chern classes of rational homogeneous
varieties will soon be discovered. Until these relations
are known, applications of these results will be
restricted to low dimensions. There is as yet no
way to see if a compact K\"ahler manifold
has a holomorphic parabolic
geometry, other than to write one down. It is possible
that the relations on characteristic classes are
necessary and sufficient conditions.

%\begin{landscape}
{\tiny{\include{ParabolicDatabase}}}
%\end{landscape}

\bibliographystyle{amsplain}
\bibliography{ChernCartan}
\end{document}

%% file: ParabolicDatabase.tex
\setlongtables
\begin{longtable}[c]{cccc}
  \caption{Chern class relations of some rational homogeneous varieties.
           These are the varieties $G/P$ of dimension up to 7 with $G$ simple.
           The dimension of the variety is the last integer in the grading.
           Notation for rational homogeneous varieties: $\Fl{p,q,\dots}{r}$
           is the space of partial flags of dimensions $p,q,\dots$ in $\mathbb{C}^r$,
           $\LagFl{p,q,\dots}{2r}$ the analogous space of partial subLagrangian flags,
           $\NullFl{p,q,\dots}{r}$ the analogous space of partial null flags.}
  \label{longtable:ChernClassRelations} \\
  \toprule
  %%\head{Dimension} &
  \head{Dynkin Diagram} &
  \head{Model} &
  \head{Grading} &
  \head{Chern Class Relations} \\
  \midrule
  \endfirsthead
  \multicolumn{4}{l}{Table \ref{longtable:ChernClassRelations}: Chern class relations of some rational homogeneous varieties (continued)} \\ \\
  \toprule
  %%\head{Dimension} &
  \head{Dynkin Diagram} &
  \head{Model} &
  \head{Grading} &
  \head{Chern Class Relations} \\
  \midrule
  \endhead
  \bottomrule
  \multicolumn{4}{l}{\small \sl{continues on next page}}
  \endfoot
  \bottomrule
  \endlastfoot
%\cmidrule{1-4}
  %%1  &
  \begin{xy}
  0;/r.10pc/: % Set the scale to be quite small
  (0,7)*+{\text{Projective connection on a curve}};
  (0,0)*+{\pout}="1";
\end{xy}
  &
  $\Proj{1}$ &
  1  &
  $\begin{array}{c}
c_{1}-2\,\varepsilon \\
\end{array}
$  \\ \addlinespace[10pt]
%%\cmidrule{1-4}
%%  2  &
  \begin{xy}
  0;/r.10pc/: % Set the scale to be quite small
  (11,7)*+{\text{Projective connection on a surface}};
  (0,0)*+{\pout}="1";
  (22,0)*+{\pin}="2";
  "1"; "2" **\dir{-};
\end{xy}
  &
  $\Proj{2}$ &
  2  &
  $\begin{array}{c}
c_{1}-3\,\varepsilon \\
c_{2}-3\,\varepsilon^2 \\
\end{array}
$  \\ \addlinespace[10pt]
%%\cmidrule{1-4}
%%  3  &
  \begin{xy}
  0;/r.10pc/: % Set the scale to be quite small
  (11,7)*+{\text{Scalar 2nd order ODE}};
  (0,0)*+{\pout}="1";
  (22,0)*+{\pout}="2";
  "1"; "2" **\dir{-};
\end{xy}
  &
  $\Proj{}T\Proj{2}$ &
  2,3  &
  $\begin{array}{c}
c_{1}-2\,\varepsilon \\
c_{2}-2\,\varepsilon^2 \\
c_{3}-\varepsilon^3 \\
\end{array}
$  \\ \addlinespace[10pt]
%%\cmidrule{2-4}
%%    &
  \begin{xy}
  0;/r.10pc/: % Set the scale to be quite small
  (11,7)*+{\text{Contact path geometry on 3-fold}};
  (0,0)*+{\pin}="1";
  (22,0)*+{\pout}="2";
  (11,0)*+{>}="edge";
  {\ar@{=}; "1";"2"};
\end{xy}
  &
  $\LagFl{1}{4}$&
  2,3  &
  $\begin{array}{c}
c_{1}-4\,\varepsilon \\
c_{2}-6\,\varepsilon^2 \\
c_{3}-4\,\varepsilon^3 \\
\end{array}
$  \\ \addlinespace[10pt]
%%\cmidrule{2-4}
%%    &
  \begin{xy}
  0;/r.10pc/: % Set the scale to be quite small
  (11,7)*+{\text{Conformal geometry on 3-fold}};
  (0,0)*+{\pout}="1";
  (22,0)*+{\pin}="2";
  (11,0)*+{>}="edge";
  {\ar@{=}; "1";"2"};
\end{xy}
  &
  $Q^3$ &
  3  &
  $\begin{array}{c}
c_{1}-3\,\varepsilon \\
c_{2}-4\,\varepsilon^2 \\
c_{3}-2\,\varepsilon^3 \\
\end{array}
$  \\ \addlinespace[10pt]
%%\cmidrule{2-4}
%%    &
  \begin{xy}
  0;/r.10pc/: % Set the scale to be quite small
  (22,7)*+{\text{Projective connection on 3-fold}};
  (0,0)*+{\pout}="1";
  (22,0)*+{\pin}="2";
  (44,0)*+{\pin}="3";
  "1"; "2" **\dir{-};
  "2"; "3" **\dir{-};
\end{xy}
  &
  $\Proj{3}$ &
  3  &
  $\begin{array}{c}
c_{1}-4\,\varepsilon \\
c_{2}-6\,\varepsilon^2 \\
c_{3}-4\,\varepsilon^3 \\
\end{array}
$  \\ \addlinespace[10pt]
%%\cmidrule{1-4}
%%  4  &
  \begin{xy}
  0;/r.10pc/: % Set the scale to be quite small
  (11,7)*+{\text{Scalar 3rd order ODE}};
  (0,0)*+{\pout}="1";
  (22,0)*+{\pout}="2";
  (11,0)*+{>}="edge";
  {\ar@{=}; "1";"2"};
\end{xy}
  &
  $\LagFl{1,2}{4}$ &
  2,3,4  &
  $\begin{array}{c}
c_{1}-2\,\varepsilon \\
c_{2}-2\,\varepsilon^2 \\
6\,c_{3}\,\varepsilon-7\,\varepsilon^4 \\
3\,c_{4}-\varepsilon^4 \\
\end{array}
$  \\ \addlinespace[10pt]
%%\cmidrule{2-4}
%%    &
  \begin{xy}
  0;/r.10pc/: % Set the scale to be quite small
  (22,7)*+{\text{Grassmann geometry on 4-fold}};
  (0,0)*+{\pin}="1";
  (22,0)*+{\pout}="2";
  (44,0)*+{\pin}="3";
  "1"; "2" **\dir{-};
  "2"; "3" **\dir{-};
\end{xy}
  &
  $\Gr{2}{4}$ &
  4  &
  $\begin{array}{c}
c_{1}-4\,\varepsilon \\
c_{2}-7\,\varepsilon^2 \\
c_{3}-6\,\varepsilon^3 \\
c_{4}-3\,\varepsilon^4 \\
\end{array}
$  \\ \addlinespace[10pt]
%%\cmidrule{2-4}
%%    &
  \begin{xy}
  0;/r.10pc/: % Set the scale to be quite small
  (33,7)*+{\text{Projective connection on 4-fold}};
  (0,0)*+{\pout}="1";
  (22,0)*+{\pin}="2";
  (44,0)*+{\pin}="3";
  (66,0)*+{\pin}="4";
  "1"; "2" **\dir{-};
  "2"; "3" **\dir{-};
  "3"; "4" **\dir{-};
\end{xy}
  &
  $\Proj{4}$ &
  4  &
  $\begin{array}{c}
c_{1}-5\,\varepsilon \\
c_{2}-10\,\varepsilon^2 \\
c_{3}-10\,\varepsilon^3 \\
c_{4}-5\,\varepsilon^4 \\
\end{array}
$  \\ \addlinespace[10pt]
%%\cmidrule{1-4}
%%  5  &
  \begin{xy}
  0;/r.10pc/: % Set the scale to be quite small
  (0,0)*+{\pin}="1";
  (22,0)*+{\pout}="2";
  (11,0)*+{<}="edge";
  "1"; "2" **\dir3{-};
\end{xy}
  &
  $Q^5$ &
  4,5  &
  $\begin{array}{c}
c_{1}-3\,\varepsilon \\
3\,c_{2}-13\,\varepsilon^2 \\
3\,c_{3}-11\,\varepsilon^3 \\
3\,c_{4}-5\,\varepsilon^4 \\
3\,c_{5}-\varepsilon^5 \\
\end{array}
$  \\ \addlinespace[10pt]
%%\cmidrule{2-4}
%%    &
  \begin{xy}
  0;/r.10pc/: % Set the scale to be quite small
  (11,7)*+{\text{Nondegenerate 2-plane field on 5-fold}};
  (0,0)*+{\pout}="1";
  (22,0)*+{\pin}="2";
  (11,0)*+{<}="edge";
  "1"; "2" **\dir3{-};
\end{xy}
  &
  &
  2,3,5  &
  $\begin{array}{c}
c_{1}-5\,\varepsilon \\
c_{2}-11\,\varepsilon^2 \\
c_{3}-13\,\varepsilon^3 \\
c_{4}-9\,\varepsilon^4 \\
c_{5}-3\,\varepsilon^5 \\
\end{array}
$  \\ \addlinespace[10pt]
%%\cmidrule{2-4}
%%    &
  \begin{xy}
  0;/r.10pc/: % Set the scale to be quite small
  (0,0)*+{\pout}="1";
  (22,0)*+{\pin}="2";
  (44,0)*+{\pout}="3";
  "1"; "2" **\dir{-};
  "2"; "3" **\dir{-};
\end{xy}
  &
  $\Fl{1,3}{4}$&
  4,5  &
  $\begin{array}{c}
c_{1}-3\,\varepsilon \\
5\,c_{2}\,\varepsilon^2-21\,\varepsilon^4 \\
c_{3}-2\,c_{2}\,\varepsilon+5\,\varepsilon^3 \\
5\,c_{4}-9\,\varepsilon^4 \\
5\,c_{5}-3\,\varepsilon^5 \\
5\,c_{2}^2-89\,\varepsilon^4 \\
\end{array}
$  \\ \addlinespace[10pt]
%%\cmidrule{2-4}
%%    &
  \begin{xy}
  0;/r.10pc/: % Set the scale to be quite small
  (22,7)*+{\text{2nd order ODE, 2 independent variables}};
  (0,0)*+{\pout}="1";
  (22,0)*+{\pout}="2";
  (44,0)*+{\pin}="3";
  "1"; "2" **\dir{-};
  "2"; "3" **\dir{-};
\end{xy}
  &
  $\Proj{}T\Proj{3}$ &
  3,5  &
  $\begin{array}{c}
375\,c_{1}^3\,c_{2}-179\,c_{1}^5 \\
-51\,c_{1}^2\,c_{2}+17\,c_{1}^4+54\,c_{1}\,c_{3} \\
18\,c_{4}-3\,c_{1}^2\,c_{2}+c_{1}^4 \\
375\,c_{5}-c_{1}^5 \\
81\,c_{2}^2-66\,c_{1}^2\,c_{2}+13\,c_{1}^4 \\
1125\,c_{2}\,c_{3}-73\,c_{1}^5 \\
\end{array}
$  \\ \addlinespace[10pt]
%%\cmidrule{2-4}
%%    &
  \begin{xy}
  0;/r.10pc/: % Set the scale to be quite small
  (22,7)*+{\text{Conformal geometry on 5-fold}};
  (0,0)*+{\pout}="1";
  (22,0)*+{\pin}="2";
  (44,0)*+{\pin}="3";
  (33,0)*+{>}="edge";
  "1"; "2" **\dir{-};
  {\ar@{=}; "2";"3"};
\end{xy}
  &
  $Q^5$ &
  5  &
  $\begin{array}{c}
c_{1}-5\,\varepsilon \\
c_{2}-11\,\varepsilon^2 \\
c_{3}-13\,\varepsilon^3 \\
c_{4}-9\,\varepsilon^4 \\
c_{5}-3\,\varepsilon^5 \\
\end{array}
$  \\ \addlinespace[10pt]
%%\cmidrule{2-4}
%%    &
  \begin{xy}
  0;/r.10pc/: % Set the scale to be quite small
  (22,7)*+{\text{Contact path geometry on 5-fold}};
  (0,0)*+{\pout}="1";
  (22,0)*+{\pin}="2";
  (44,0)*+{\pin}="3";
  (33,0)*+{<}="edge";
  "1"; "2" **\dir{-};
  {\ar@{=}; "2";"3"};
\end{xy}
  &
  $\LagFl{1}{6}$&
  4,5  &
  $\begin{array}{c}
c_{1}-6\,\varepsilon \\
c_{2}-15\,\varepsilon^2 \\
c_{3}-20\,\varepsilon^3 \\
c_{4}-15\,\varepsilon^4 \\
c_{5}-6\,\varepsilon^5 \\
\end{array}
$  \\ \addlinespace[10pt]
%%\cmidrule{2-4}
%%    &
  \begin{xy}
  0;/r.10pc/: % Set the scale to be quite small
  (44,7)*+{\text{Projective connection on 5-fold}};
  (0,0)*+{\pout}="1";
  (22,0)*+{\pin}="2";
  (44,0)*+{\pin}="3";
  (66,0)*+{\pin}="4";
  (88,0)*+{\pin}="5";
  "1"; "2" **\dir{-};
  "2"; "3" **\dir{-};
  "3"; "4" **\dir{-};
  "4"; "5" **\dir{-};
\end{xy}
  &
  $\Proj{5}$ &
  5  &
  $\begin{array}{c}
c_{1}-6\,\varepsilon \\
c_{2}-15\,\varepsilon^2 \\
c_{3}-20\,\varepsilon^3 \\
c_{4}-15\,\varepsilon^4 \\
c_{5}-6\,\varepsilon^5 \\
\end{array}
$  \\ \addlinespace[10pt]
%%\cmidrule{1-4}
%%  6  &
  \begin{xy}
  0;/r.10pc/: % Set the scale to be quite small
  (0,0)*+{\pout}="1";
  (22,0)*+{\pout}="2";
  (11,0)*+{<}="edge";
  "1"; "2" **\dir3{-};
\end{xy}
  &
  &
  2,3,4,5,6  &
  $\begin{array}{c}
c_{1}-2\,\varepsilon \\
c_{2}-2\,\varepsilon^2 \\
180\,c_{3}\,\varepsilon^3-227\,\varepsilon^6 \\
c_{4}-2\,c_{3}\,\varepsilon+2\,\varepsilon^4 \\
49\,c_{5}-48\,c_{3}\,\varepsilon^2+54\,\varepsilon^5 \\
60\,c_{6}-\varepsilon^6 \\
90\,c_{3}^2-143\,\varepsilon^6 \\
\end{array}
$  \\ \addlinespace[10pt]
%%\cmidrule{2-4}
%%    &
  \begin{xy}
  0;/r.10pc/: % Set the scale to be quite small
  (0,0)*+{\pout}="1";
  (22,0)*+{\pout}="2";
  (44,0)*+{\pout}="3";
  "1"; "2" **\dir{-};
  "2"; "3" **\dir{-};
\end{xy}
  &
  $\Fl{1,2,3}{4}$ &
  3,5,6  &
  $\begin{array}{c}
c_{1}-2\,\varepsilon \\
c_{2}-2\,\varepsilon^2 \\
18\,c_{3}\,\varepsilon^3-23\,\varepsilon^6 \\
c_{4}-2\,c_{3}\,\varepsilon+2\,\varepsilon^4 \\
5\,c_{5}-3\,c_{3}\,\varepsilon^2+3\,\varepsilon^5 \\
30\,c_{6}-\varepsilon^6 \\
45\,c_{3}^2-73\,\varepsilon^6 \\
\end{array}
$  \\ \addlinespace[10pt]
%%\cmidrule{2-4}
%%    &
  \begin{xy}
  0;/r.10pc/: % Set the scale to be quite small
  (0,0)*+{\pin}="1";
  (22,0)*+{\pin}="2";
  (44,0)*+{\pout}="3";
  (33,0)*+{>}="edge";
  (33,7)*+{\text{Nondegenerate 3-plane field on 6-fold}};
  "1"; "2" **\dir{-};
  {\ar@{=}; "2";"3"};
\end{xy}
  &
  \NullFl{3}{7} &
  3,6  &
  $\begin{array}{c}
c_{1}-6\,\varepsilon \\
c_{2}-16\,\varepsilon^2 \\
c_{3}-24\,\varepsilon^3 \\
c_{4}-22\,\varepsilon^4 \\
c_{5}-12\,\varepsilon^5 \\
c_{6}-4\,\varepsilon^6 \\
\end{array}
$  \\ \addlinespace[10pt]
%%\cmidrule{2-4}
%%    &
  \begin{xy}
  0;/r.10pc/: % Set the scale to be quite small
  (0,0)*+{\pin}="1";
  (22,0)*+{\pin}="2";
  (44,0)*+{\pout}="3";
  (33,0)*+{<}="edge";
  "1"; "2" **\dir{-};
  {\ar@{=}; "2";"3"};
\end{xy}
  &
  $\LagFl{3}{6}$ &
  6  &
  $\begin{array}{c}
c_{1}-4\,\varepsilon \\
2\,c_{2}-15\,\varepsilon^2 \\
8\,c_{3}\,\varepsilon-67\,\varepsilon^4 \\
4\,c_{4}-23\,\varepsilon^4 \\
4\,c_{5}-9\,\varepsilon^5 \\
2\,c_{6}-\varepsilon^6 \\
8\,c_{3}^2-555\,\varepsilon^6 \\
\end{array}
$  \\ \addlinespace[10pt]
%%\cmidrule{2-4}
%%    &
  \begin{xy}
  0;/r.10pc/: % Set the scale to be quite small
  (33,7)*+{\text{Grassmann geometry on 6-fold}};
  (0,0)*+{\pin}="1";
  (22,0)*+{\pout}="2";
  (44,0)*+{\pin}="3";
  (66,0)*+{\pin}="4";
  "1"; "2" **\dir{-};
  "2"; "3" **\dir{-};
  "3"; "4" **\dir{-};
\end{xy}
  &
  $\Gr{2}{5}$ &
  6  &
  $\begin{array}{c}
c_{1}-5\,\varepsilon \\
5\,\varepsilon^3\,c_{2}-57\,\varepsilon^5 \\
c_{3}-15\,\varepsilon^3 \\
c_{4}-5\,c_{2}\,\varepsilon^2+45\,\varepsilon^4 \\
c_{5}-6\,\varepsilon^5 \\
c_{6}-2\,\varepsilon^6 \\
c_{2}^2-25\,c_{2}\,\varepsilon^2+155\,\varepsilon^4 \\
\end{array}
$  \\ \addlinespace[10pt]
%%\cmidrule{2-4}
%%    &
  \begin{xy}
  0;/r.10pc/: % Set the scale to be quite small
  (18,26)*+{\text{Conformal geometry on 6-fold}};
  (0,0)*+{\pout}="1";
  (22,0)*+{\pin}="2";
  (33,19)*+{\pin}="3";
  (33,-19)*+{\pin}="4";
  "1"; "2" **\dir{-};
  "2"; "3" **\dir{-};
  "2"; "4" **\dir{-};
\end{xy}
  &
  $Q^6$ &
  6  &
  $\begin{array}{c}
c_{1}-6\,\varepsilon \\
c_{2}-16\,\varepsilon^2 \\
c_{3}-24\,\varepsilon^3 \\
c_{4}-22\,\varepsilon^4 \\
c_{5}-12\,\varepsilon^5 \\
c_{6}-4\,\varepsilon^6 \\
\end{array}
$  \\ \addlinespace[10pt]
%%\cmidrule{2-4}
%%    &
  \begin{xy}
  0;/r.10pc/: % Set the scale to be quite small
  (55,7)*+{\text{Projective connection on 6-fold}};
  (0,0)*+{\pout}="1";
  (22,0)*+{\pin}="2";
  (44,0)*+{\pin}="3";
  (66,0)*+{\pin}="4";
  (88,0)*+{\pin}="5";
  (110,0)*+{\pin}="6";
  "1"; "2" **\dir{-};
  "2"; "3" **\dir{-};
  "3"; "4" **\dir{-};
  "4"; "5" **\dir{-};
  "5"; "6" **\dir{-};
\end{xy}
  &
  $\Proj{6}$ &
  6  &
  $\begin{array}{c}
c_{1}-7\,\varepsilon \\
c_{2}-21\,\varepsilon^2 \\
c_{3}-35\,\varepsilon^3 \\
c_{4}-35\,\varepsilon^4 \\
c_{5}-21\,\varepsilon^5 \\
c_{6}-7\,\varepsilon^6 \\
\end{array}
$  \\ \addlinespace[10pt]
%%\cmidrule{1-4}
%%  7  &
  \begin{xy}
  0;/r.10pc/: % Set the scale to be quite small
  (0,0)*+{\pin}="1";
  (22,0)*+{\pout}="2";
  (44,0)*+{\pin}="3";
  (33,0)*+{>}="edge";
  "1"; "2" **\dir{-};
  {\ar@{=}; "2";"3"};
\end{xy}
  &
  $\NullFl{2}{7}$ &
  6,7  &
  $\begin{array}{c}
c_{1}-4\,\varepsilon \\
14\,\varepsilon^4\,c_{2}-107\,\varepsilon^6 \\
c_{3}-3\,c_{2}\,\varepsilon+14\,\varepsilon^3 \\
3\,c_{4}+c_{2}\,\varepsilon^2-28\,\varepsilon^4 \\
3\,c_{5}+17\,\varepsilon^3\,c_{2}-140\,\varepsilon^5 \\
14\,c_{6}-15\,\varepsilon^6 \\
14\,c_{7}-3\,\varepsilon^7 \\
3\,c_{2}^2-44\,c_{2}\,\varepsilon^2+161\,\varepsilon^4 \\
\end{array}
$  \\ \addlinespace[10pt]
%%\cmidrule{2-4}
%%    &
  \begin{xy}
  0;/r.10pc/: % Set the scale to be quite small
  (0,0)*+{\pin}="1";
  (22,0)*+{\pout}="2";
  (44,0)*+{\pin}="3";
  (33,0)*+{<}="edge";
  (33,14)*+{\text{Nondegenerate 4-plane field on 7-fold,}};
  (33,7)*+{\text{a.k.a. quaternionic contact structure}};
  "1"; "2" **\dir{-};
  {\ar@{=}; "2";"3"};
\end{xy}
  &
  $\LagFl{2}{6}$ &
  4,7  &
  $\begin{array}{c}
c_{1}-5\,\varepsilon \\
7\,\varepsilon^4\,c_{2}-82\,\varepsilon^6 \\
c_{3}-2\,c_{2}\,\varepsilon+7\,\varepsilon^3 \\
3\,c_{4}-5\,c_{2}\,\varepsilon^2+14\,\varepsilon^4 \\
3\,c_{5}-10\,\varepsilon^3\,c_{2}+91\,\varepsilon^5 \\
7\,c_{6}-24\,\varepsilon^6 \\
7\,c_{7}-6\,\varepsilon^7 \\
3\,c_{2}^2-74\,c_{2}\,\varepsilon^2+455\,\varepsilon^4 \\
\end{array}
$  \\ \addlinespace[10pt]
%%\cmidrule{2-4}
%%    &
  \begin{xy}
  0;/r.10pc/: % Set the scale to be quite small
  (0,0)*+{\pout}="1";
  (22,0)*+{\pin}="2";
  (44,0)*+{\pin}="3";
  (66,0)*+{\pout}="4";
  "1"; "2" **\dir{-};
  "2"; "3" **\dir{-};
  "3"; "4" **\dir{-};
\end{xy}
  &
  $\Fl{1,4}{5}$&
  6,7  &
  $\begin{array}{c}
c_{1}-4\,\varepsilon \\
7\,\varepsilon^4\,c_{2}-52\,\varepsilon^6 \\
c_{3}-3\,c_{2}\,\varepsilon+14\,\varepsilon^3 \\
c_{4}-5\,c_{2}\,\varepsilon^2+31\,\varepsilon^4 \\
c_{5}-5\,\varepsilon^3\,c_{2}+34\,\varepsilon^5 \\
7\,c_{6}-8\,\varepsilon^6 \\
7\,c_{7}-2\,\varepsilon^7 \\
c_{2}^2-17\,c_{2}\,\varepsilon^2+71\,\varepsilon^4 \\
\end{array}
$  \\ \addlinespace[10pt]
%%\cmidrule{2-4}
%%    &
  \begin{xy}
  0;/r.10pc/: % Set the scale to be quite small
  (0,0)*+{\pout}="1";
  (22,0)*+{\pout}="2";
  (44,0)*+{\pin}="3";
  (66,0)*+{\pin}="4";
  "1"; "2" **\dir{-};
  "2"; "3" **\dir{-};
  "3"; "4" **\dir{-};
\end{xy}
  &
  $\Fl{1,2}{5}$ &
  4,7  &
  $\begin{array}{c}
c_{1}-2\,\varepsilon \\
126\,c_{2}\,\varepsilon^5-239\,\varepsilon^7 \\
-3942\,\varepsilon^4\,c_{2}+2160\,c_{3}\,\varepsilon^3+5083\,\varepsilon^6 \\
768\,c_{4}\,\varepsilon-816\,c_{3}\,\varepsilon^2+414\,\varepsilon^3\,c_{2}-215\,\varepsilon^5 \\
192\,c_{5}-48\,c_{3}\,\varepsilon^2-18\,\varepsilon^3\,c_{2}+65\,\varepsilon^5 \\
540\,c_{6}-54\,\varepsilon^4\,c_{2}+91\,\varepsilon^6 \\
378\,c_{7}-\varepsilon^7 \\
c_{2}^2-c_{4}+2\,c_{3}\,\varepsilon-6\,c_{2}\,\varepsilon^2+6\,\varepsilon^4 \\
1536\,c_{2}\,c_{3}-1104\,c_{3}\,\varepsilon^2-4878\,\varepsilon^3\,c_{2}+7247\,\varepsilon^5 \\
4320\,c_{2}\,c_{4}-12582\,\varepsilon^4\,c_{2}+20303\,\varepsilon^6 \\
4320\,c_{3}^2-16794\,\varepsilon^4\,c_{2}+26561\,\varepsilon^6 \\
378\,c_{3}\,c_{4}-181\,\varepsilon^7 \\
\end{array}
$  \\ \addlinespace[10pt]
%%\cmidrule{2-4}
%%    &
  \begin{xy}
  0;/r.10pc/: % Set the scale to be quite small
  (33,7)*+{\text{Conformal geometry on 7-fold}};
  (0,0)*+{\pout}="1";
  (22,0)*+{\pin}="2";
  (44,0)*+{\pin}="3";
  (66,0)*+{\pin}="4";
  (55,0)*+{>}="edge";
  "1"; "2" **\dir{-};
  "2"; "3" **\dir{-};
  {\ar@{=}; "3";"4"};
\end{xy}
  &
  $Q^7$ &
  7  &
  $\begin{array}{c}
c_{1}-7\,\varepsilon \\
c_{2}-22\,\varepsilon^2 \\
c_{3}-40\,\varepsilon^3 \\
c_{4}-46\,\varepsilon^4 \\
c_{5}-34\,\varepsilon^5 \\
c_{6}-16\,\varepsilon^6 \\
c_{7}-4\,\varepsilon^7 \\
\end{array}
$  \\ \addlinespace[10pt]
%%\cmidrule{2-4}
%%    &
  \begin{xy}
  0;/r.10pc/: % Set the scale to be quite small
  (33,7)*+{\text{Contact path geometry on 7-fold}};
  (0,0)*+{\pout}="1";
  (22,0)*+{\pin}="2";
  (44,0)*+{\pin}="3";
  (66,0)*+{\pin}="4";
  (55,0)*+{<}="edge";
  "1"; "2" **\dir{-};
  "2"; "3" **\dir{-};
  {\ar@{=}; "3";"4"};
\end{xy}
  &
  $\LagFl{1}{8}$ &
  6,7  &
  $\begin{array}{c}
c_{1}-8\,\varepsilon \\
c_{2}-28\,\varepsilon^2 \\
c_{3}-56\,\varepsilon^3 \\
c_{4}-70\,\varepsilon^4 \\
c_{5}-56\,\varepsilon^5 \\
c_{6}-28\,\varepsilon^6 \\
c_{7}-8\,\varepsilon^7 \\
\end{array}
$  \\ \addlinespace[10pt]
%%\cmidrule{2-4}
%%    &
  \begin{xy}
  0;/r.10pc/: % Set the scale to be quite small
  (66,7)*+{\text{Projective connection on 7-fold}};
  (0,0)*+{\pout}="1";
  (22,0)*+{\pin}="2";
  (44,0)*+{\pin}="3";
  (66,0)*+{\pin}="4";
  (88,0)*+{\pin}="5";
  (110,0)*+{\pin}="6";
  (132,0)*+{\pin}="7";
  "1"; "2" **\dir{-};
  "2"; "3" **\dir{-};
  "3"; "4" **\dir{-};
  "4"; "5" **\dir{-};
  "5"; "6" **\dir{-};
  "6"; "7" **\dir{-};
\end{xy}
  &
  $\Proj{7}$ &
  7  &
  $\begin{array}{c}
c_{1}-8\,\varepsilon \\
c_{2}-28\,\varepsilon^2 \\
c_{3}-56\,\varepsilon^3 \\
c_{4}-70\,\varepsilon^4 \\
c_{5}-56\,\varepsilon^5 \\
c_{6}-28\,\varepsilon^6 \\
c_{7}-8\,\varepsilon^7 \\
\end{array}
$  \\
\end{longtable}